\documentclass[11pt]{amsart}
\usepackage{amsmath}
\usepackage{a4wide}
\usepackage[utf8]{inputenc}
\usepackage{amssymb}
\usepackage{amsopn}
\usepackage{epsfig}
\usepackage{amsfonts}
\usepackage{latexsym}
\usepackage{amsthm}
\usepackage{enumerate}
\usepackage[UKenglish]{babel}
\usepackage{verbatim}
\usepackage{color}
\usepackage{pgf}
\usepackage{tikz}

\usepackage{mathrsfs}   % This provides an alternative way to typeset script capital letters, using \mathscr{}
\usetikzlibrary{arrows,automata}

\newtheorem{theorem}{Theorem}[section]
\newtheorem{lemma}[theorem]{Lemma}
\newtheorem{proposition}[theorem]{Proposition}
\newtheorem{corollary}[theorem]{Corollary}
\newtheorem{main}{Theorem}

\theoremstyle{definition}
\newtheorem{definition}[theorem]{Definition}
\newtheorem{example}[theorem]{Example}

\theoremstyle{remark}
\newtheorem{remark}[theorem]{Remark}

\numberwithin{equation}{section}

\newcommand{\R}{\ensuremath{\mathbb{R}}}
\newcommand{\N}{\ensuremath{\mathbb{N}}}

\newcommand{\bb}{\mathscr{B}}

\newcommand{\cb}{ {\mathscr{C}}}

\newcommand{\was}{\widetilde{\mathbf{U}}^*}

\renewcommand{\u}{\ensuremath{\mathcal{U}}}
\newcommand{\ub}{\mathscr{U}}
\newcommand{\us}{\mathbf{U}}
\newcommand{\wus}{\widetilde{\mathbf{U}}}
 
\newcommand{\ws}{\mathbf{W}}
 
 \newcommand{\wvs}{\widetilde{\mathbf{V}}}
 \newcommand{\vb}{\mathscr{V}}
\newcommand{\vs}{\mathbf{V}}

\newcommand{\sv}{\mathbf{a}}
\renewcommand{\sb}{\mathbf{b}}
\newcommand{\sd}{\mathbf{d}}
\renewcommand{\sc}{\mathbf{c}}

\newcommand{\set}[1]{\left\{#1\right\}}
\newcommand{\la}{\lambda}

\newcommand{\ep}{\varepsilon}
\newcommand{\f}{\infty}
\newcommand{\de}{\alpha^*}

\newcommand{\al}{\alpha}
\newcommand{\lle}{\preccurlyeq}
\newcommand{\lge}{\succcurlyeq}
\newcommand{\si}{\sigma}

\newcommand{\ra}{\rightarrow}

\newcommand{\diam}{\operatorname{diam}}

\begin{document}

\title{Relative bifurcation sets and the local dimension of univoque bases}

\author{Pieter Allaart}
\address[P. Allaart]{Mathematics Department, University of North Texas, 1155 Union Cir \#311430, Denton, TX 76203-5017, U.S.A.}
\email{allaart@unt.edu}

\author{Derong Kong}
\address[D. Kong]{College of Mathematics and Statistics, Chongqing University, 401331, Chongqing, P.R.China}
\email{derongkong@126.com}

%\date{\today}
\dedicatory{}

%\begin{frontmatter}


\begin{abstract}
Fix an alphabet $A=\{0,1,\dots,M\}$ with $M\in\mathbb{N}$. The univoque set $\mathscr{U}$ of bases $q\in(1,M+1)$ in which the number $1$ has a unique expansion over the alphabet $A$ has been well studied. It has Lebesgue measure zero but Hausdorff dimension one. This paper investigates how the set $\mathscr{U}$ is distributed over the interval $(1,M+1)$ by determining the limit
$$f(q):=\lim_{\delta\to 0}\dim_H\big(\mathscr{U}\cap(q-\delta,q+\delta)\big)$$
for all $q\in(1,M+1)$. We show in particular that $f(q)>0$ if and only if $q\in\overline{\mathscr{U}}\backslash\mathscr{C}$, where $\mathscr{C}$ is an uncountable set of Hausdorff dimension zero, and $f$ is continuous at those (and only those) points where it vanishes. Furthermore, we introduce a countable family of pairwise disjoint subsets of $\mathscr{U}$ called {\em relative bifurcation sets}, and use them to  give an explicit expression for the Hausdorff dimension of the intersection of $\mathscr{U}$ with any interval, answering a question of Kalle et al.~[{\em arXiv:1612.07982; to appear in Acta Arithmetica}, 2018]. Finally, the methods developed in this paper are used to give a complete answer to a question of the first author [{\em Adv. Math.}, 308:575--598, 2017] about strongly univoque sets.
\end{abstract}


\maketitle

\section{Introduction} \label{s1}

Fix an integer $M\ge 1$. For $q\in(1,M+1]$, any real number $x$ in the interval $I_{M,q}:=[0, M/(q-1)]$ can be represented as 
\begin{equation} \label{eq:projection-pi-q}
x=\pi_q((d_i)):=\sum_{i=1}^\f\frac{d_i}{q^i},
\end{equation}
where $d_i\in\set{0,1,\ldots, M}$ for all $i\ge 1.$
The infinite sequence $(d_i)=d_1d_2\ldots$ is called a \emph{$q$-expansion} of $x$ with \emph{alphabet} $\set{0,1,\ldots, M}$. %Notice that $x\in I_{M,q}$ may have multiple $q$-expansions. 
Such non-integer base expansions have been studied since the pioneering work of R\'enyi \cite{Renyi_1957} and Parry \cite{Parry_1960}. In the 1990's, work by Erd\H os et al.~\cite{Erdos_Joo_Komornik_1990, Erdos_Horvath_Joo_1991, Erdos_Joo_1992} inspired an explosion of research papers on the subject, covering unique expansions \cite{AlcarazBarrera-Baker-Kong-2016, DeVries_Komornik_2008, Glendinning_Sidorov_2001,  Komornik-Kong-Li-17}, finitely or countably many expansions \cite{Baker_2015, Baker_Sidorov_2014, Komornik-Kong-2018, Sidorov_2009}, uncountably many expansions and random expansions \cite{Dajani_DeVries_2007, Sidorov_2003}. Non-integer base expansions have furthermore been connected with Bernoulli convolutions \cite{Jordan-Shmerkin-Solomyak-2011}, Diophantine approximation \cite{Lu-Wu-2016}, singular self-affine functions \cite{Allaart-2016}, open dynamical systems \cite{Sidorov_2007}, and intersections of Cantor sets \cite{Kong_Li_Dekking_2010}.

Let 
\begin{equation*} %\label{eq:univoque bases}
\ub:=\set{q\in(1, M+1]: 1\textrm{ has a unique }q\textrm{-expansion of the form } (\ref{eq:projection-pi-q})}.
\end{equation*}
Thus for each $q\in\ub$ there exists a unique sequence $(a_i)\in\Omega_M:=\set{0,1,\ldots, M}^\N$ such that $1=\pi_q((a_i))$. 
The set $\ub$ was extensively studied for over 25 years.  Erd\H os et al.~\cite{Erdos_Joo_Komornik_1990} showed that $\ub$ is uncountable and of zero Lebesgue measure. Dar\'oczy and K\'atai \cite{Darczy_Katai_1995} proved that $\ub$ has full Hausdorff dimension (see also \cite{Komornik-Kong-Li-17}). Komornik and Loreti \cite{Komornik-Loreti-1998, Komornik_Loreti_2002} found its smallest element $q_{KL}=q_{KL}(M)$, which is now called the \emph{Komornik-Loreti constant} and is related to the Thue-Morse sequence (see (\ref{eq:lambda}) below). Later in \cite{Komornik_Loreti_2007} the same authors proved that its topological closure $\overline{\ub}$ is a Cantor set, i.e., a non-empty compact set having neither interior nor isolated points. Recently, Dajani et al.~\cite{Dajani-Komornik-Kong-Li-2018} proved that the algebraic difference $\ub-\ub$ contains an interval. Furthermore,  the set $\ub$ also has intimate connections with kneading sequences of unimodal expanding maps (cf.~\cite{Allouche_Cosnard_1983, Allouche-Cosnard-2001}), and even with the real slice of the boundary of the Mandelbrot set \cite{Bon-Car-Ste-Giu-2013}. 

The main purpose of this paper is to describe the distribution of $\ub$. More precisely, we are interested in the \emph{local dimensional function}
\[
f(q):=\lim_{\delta\ra 0}\dim_H(\ub\cap(q-\delta, q+\delta)), \qquad q\in(1, M+1],
\]
as well as its one-sided analogs
\[
f_-(q):=\lim_{\delta\ra 0}\dim_H (\ub\cap(q-\delta, q)),\qquad f_+(q):=\lim_{\delta\ra 0}\dim_H(\ub\cap(q, q+\delta)),
\]
which we call the \emph{left and right local dimensional functions} of $\ub$. Note that $f=\max\set{f_-, f_+}$, and if $q\notin\overline{\ub}$, then $f(q)=f_-(q)=f_+(q)=0$. 
{Extending a recent result by the authors and Baker \cite{Allaart-Baker-Kong-17},} we compute $f(q)$, $f_-(q)$ and $f_+(q)$ for every $q\in(1,M+1]$ in terms of a kind of localized entropy. As an application we compute the Hausdorff dimension of the intersection of $\ub$ with any interval, answering a question of Kalle et al.~\cite{Kalle-Kong-Li-Lv-2016}. In addition, our methods allow us to give a complete answer to a question of the first author \cite{Allaart-2017} about strongly univoque sets.

\subsection{Univoque set, entropy plateaus and the bifurcation set}
 
In order to state our main results, some notation is necessary.
For $q\in(1,M+1]$ let $\u_q$ be the \emph{univoque set} of $x\in I_{M,q}$ having a unique $q$-expansion as in (\ref{eq:projection-pi-q}). Let $\us_q$ be the set of corresponding sequences, i.e., 
\[
\us_q:=\set{(d_i)\in\Omega_M: \pi_q((d_i))\in\u_q}.
\]
A useful tool in the study of unique expansions is the lexicographical characterization of $\us_q$ (cf.~\cite{Baiocchi_Komornik_2007, DeVries_Komornik_2008}):
$(d_i)\in\us_q$ if and only if $(d_i)\in\Omega_M$ satisfies
\begin{equation}\label{eq:characterization-unique expansion}
\begin{split}
d_{n+1}d_{n+2}\ldots&\prec \al(q)\qquad\textrm{if}\quad d_n<M,\\
d_{n+1}d_{n+2}\ldots&\succ\overline{\al(q)}\qquad\textrm{if}\quad d_n>0,
\end{split}
\end{equation}
where $\al(q)=(\al_i(q))\in\Omega_M$ is the lexicographically largest $q$-expansion of $1$ not ending with $0^\f$, called the \emph{quasi-greedy} $q$-expansion of $1$, and $\overline{\al(q)}:=(M-\al_i(q))$. Here and throughout the paper we will use the lexicographical order between sequences and blocks in a natural way. 

Note by (\ref{eq:characterization-unique expansion}) that any sequence $(d_i)\in\us_q\setminus\set{0^\f, M^\f}$ has a tail sequence in the set
\begin{equation} \label{eq:def-widetilde-uq}
\wus_q:=\set{(d_i)\in\Omega_M: \overline{\al(q)}\prec\si^n((d_i))\prec \al(q)\ \forall n\ge 0},
\end{equation}
where $\sigma$ denotes the left shift map on $\Omega_M$.
Furthermore, $\us_q$ and $\wus_q$ have the same topological entropy, i.e., $h(\us_q)=h(\wus_q)$, where the \emph{topological entropy} of a subset $X\subset\Omega_M$ is defined by
\[
h(X):=\liminf_{n\ra\f}\frac{\log \# B_n(X)}{n}
\]
(cf.~\cite{Lind_Marcus_1995}). Here $\#B_n(X)$ denotes the number of all length $n$ blocks occurring in sequences from $X$, and ``$\log$" denotes the natural logarithm. We may thus obtain all the relevant information about $\us_q$ by studying the simpler set $\wus_q$.

Since the map $q\mapsto \al(q)$ is strictly increasing on $(1, M+1]$ (see Lemma \ref{lem:quasi-greedy expansion-alpha-q} below), (\ref{eq:def-widetilde-uq}) implies that the set-valued map $q\mapsto \wus_q$ is non-decreasing, and hence the entropy function $H: q\mapsto h(\wus_q)$ is non-decreasing. Recently, Komornik et al.~\cite{Komornik-Kong-Li-17} and the present authors \cite{Allaart-Kong-2018} proved the following:

\begin{theorem}[\cite{Komornik-Kong-Li-17,Allaart-Kong-2018}] \label{thm:devil-staircase}
The graph of $H$ is a Devil's staircase:
\begin{enumerate}[(i)] 
\item $H$ is non-decreasing and continuous on $(1, M+1]$; 
\item $H$ is locally constant almost everywhere on $(1, M+1]$;
\item $H(q)>0$ if and only if $q>q_{KL}$, where $q_{KL}$ is the {Komornik-Loreti constant}. 
\end{enumerate}
\end{theorem}

An interval $[p_L, p_R]\subset(1, M+1]$ is called an \emph{entropy plateau} (or simply, a \emph{plateau}) if it is a maximal interval (in the partial order of set inclusion) on which $H$ is constant and positive.
%By the above argument it follows that the first plateau is $(1, q_{KL}]$ in which $H=0$. 
A complete characterization of all entropy plateaus was given by Alcaraz Barrera et al.~\cite{AlcarazBarrera-Baker-Kong-2016} (see also \cite{Alcaraz_Barrera_2014} for the case $M=1$). Equivalently, they described the \emph{bifurcation set}
\begin{equation} \label{eq:bifurcation set}
\bb:=\set{q\in(1, M+1]: H(p)\ne H(q)~\forall p\ne q},
\end{equation}
and showed that $\bb\subset \ub$, $\bb$ is Lebesgue null, and $\dim_H \bb=1$. 
From Theorem \ref{thm:devil-staircase} and the definition of $\bb$ it follows that
\begin{equation} \label{eq:relation-entropy plateaus-bifurcation set}
(1, M+1]\setminus\bb=(1, q_{KL}]\cup\bigcup[p_L, p_R],
\end{equation}
where the union is taken over all plateaus $[p_L, p_R]\subset(q_{KL}, M+1]$ of $H$. We emphasize that the plateaus are pairwise disjoint and therefore the union is countable. 

{Recall that our main objective is to find the local dimensional functions $f$, $f_+$ and $f_-$. The following result is due to the authors and Baker \cite{Allaart-Baker-Kong-17}.

\begin{proposition}[Allaart, Baker and Kong \cite{Allaart-Baker-Kong-17}] \label{prop:local dimension-B}
For any $q\in\bb\setminus\set{M+1}$ we have 
\[
f(q)=f_-(q)=f_+(q)=\dim_H\u_q>0,
\]
and for any $q\in(1,M+1]$ we have $f(q)\leq\dim_H\u_q$.
Furthermore, for $q=M+1$ we have $f(q)=f_-(q)=1$ and $f_+(q)=0$. 
\end{proposition}
}

\subsection{Relative bifurcation sets and relative plateaus}
In order to describe the local dimensional function $f$ of $\ub$ we introduce the relative bifurcation sets, which provide finer information about the growth of $q\mapsto\wus_q$ inside entropy plateaus.

\begin{definition} \label{def:admissible-words}
A word $a_1\dots a_m\in\{0,1,\dots,M\}^m$ with $m\geq 2$ is {\em admissible} if
\begin{equation}\label{eq:admissible word}
\overline{a_1\ldots a_{m-i}}\lle a_{i+1}\ldots a_m\prec a_1\ldots a_{m-i}\quad\forall ~1\le i<m.
\end{equation}
When $M\geq 2$, the ``word" $a_1\in\{0,1,\dots,M\}$ is {\em admissible} if $\overline{a_1}\leq a_1<M$.
\end{definition}

For any admissible word $\sv$, there are bases $q_L$ and $q_R$ such that
\[
\alpha(q_L)=\sv^\f, \qquad \alpha(q_R)=\sv^+(\overline{\sv})^\f.
\]
Here, for a word $\mathbf c:=c_1\ldots c_n\in\set{0,1,\ldots, M}^n$ with $c_n<M$ we set $\mathbf c^+:=c_1\ldots c_{n-1}(c_n+1)$. Similarly, for a word $\mathbf c:=c_1\ldots c_n\in\set{0,1,\ldots, M}^n$ with $c_n>0$ we shall write $\mathbf c^-:=c_1\ldots c_{n-1}(c_n-1)$.
We call $[q_L,q_R]$ a {\em basic interval} and say it is {\em generated by} the word $\sv$. By \cite[Lemma 4.8]{AlcarazBarrera-Baker-Kong-2016}, any two basic intervals are either disjoint, or else one contains the other. For any basic interval $I$ generated by an admissible word $\sv$, we define the associated {\em de Vries-Komornik number} $q_c(I)$ by $\alpha(q_c(I))=(\theta_i)$ (cf. \cite{Kong_Li_2015}), where $(\theta_i)$ is given recursively by
\begin{enumerate}[(i)]
\item $\theta_1\dots \theta_m=\sv^+$;
\item $\theta_{2^{k-1}m+1}\dots\theta_{2^k m}=\overline{\theta_1\dots \theta_{2^{k-1}m}}^+$, for $k=1,2,\dots$.
\end{enumerate}
Thus, 
\begin{equation}
\alpha(q_c(I))=\sv^+\overline{\sv}\overline{\sv^+}\sv^+\overline{\sv^+}\sv\sv^+\overline{\sv}\cdots.
\label{eq:q_c}
\end{equation}
Note that $q_c(I)$ lies in the interior of $I$ for each basic interval $I$; this is a direct consequence of Lemma \ref{lem:quasi-greedy expansion-alpha-q} below. Observe also that different basic intervals can have the same associated de Vries-Komornik number.

We now construct a nested tree
\[
\set{J_{\mathbf i}: \mathbf i\in\set{1,2,\ldots}^n; n\ge 1}
\]
of intervals, which we call {\em relative entropy plateaus}, or simply {\em relative plateaus}, as follows. At level $0$, we set $J_\emptyset=[1,M+1]$. Next, at level $1$, we put $J_0=[1,q_{KL}]$ and let $J_1,J_2,\dots$ be an arbitrary enumeration of the entropy plateaus $[p_L,p_R]$ from \eqref{eq:relation-entropy plateaus-bifurcation set}. Note by \cite{AlcarazBarrera-Baker-Kong-2016} that these entropy plateaus are precisely the maximal basic intervals which lie completely to the right of $q_{KL}$. We call $J_0$ a {\em null interval}, since $\ub\cap (1,q_{KL})=\emptyset$. 

From here, we proceed inductively as follows. Let $n\geq 1$, and for each $\mathbf{i}\in\{1,2,\dots\}^n$, assume $J_{\mathbf i}$ has already been defined and is a basic interval $[q_L,q_R]$. Then we set $J_{\mathbf{i}0}=[q_L,q_c(J_{\mathbf i})]$, and let $J_{\mathbf{i}1},J_{\mathbf{i}2},\dots$ be an arbitrary enumeration of the maximal basic intervals inside $[q_c(J_{\mathbf i}),q_R]$. (It is not difficult to see that infinitely many such basic intervals exist.)

Note that for each fixed $n\ge 1$ the relative plateaus $J_{\mathbf i}, \mathbf i\in\set{1,2,\ldots}^n$ are pairwise disjoint. Furthermore, for any word $\mathbf i\in\bigcup_{n=1}^\f\set{1,2,\ldots}^n$ we call $J_{\mathbf i 0}$ a {\em null interval}, because it intersects $\ub$ only in the single point $q_c(J_{\mathbf i})$. We emphasize that any basic interval generated by a word $\sv$ not of the form $\sb\overline{\sb}$ is a relative plateau.

We now define the sets
\begin{equation*}
\cb_\f:= \bigcap_{n=1}^\f\bigcup_{\mathbf i\in\set{1,2,\ldots}^n} J_{\mathbf i}
%\label{eq:definition-of-C-infinity}
\end{equation*}
and
\[
\cb_0:=\left\{q_c(J_{\mathbf i}): \mathbf{i}\in\bigcup_{n=0}^\f\set{1,2,\ldots}^n\right\}.
\]
Thus $\cb_\f$ is the set of points which are contained in infinitely many relative plateaus, and $\cb_0$ is the set of all de Vries-Komornik numbers (cf.~\cite{Kong_Li_2015}). The smallest element of $\cb_0$ is the Komornik-Loreti constant $q_{KL}=q_c(J_\emptyset)$. Finally, let
\begin{equation*}
\cb:=\cb_0\cup\cb_\f.
%\label{eq:definition-of-C}
\end{equation*}
For the proof of the following proposition, as well as examples of points in $\cb$, we refer to Section \ref{sec:C}.

\begin{proposition} \label{prop:property of C-infity} \mbox{}
\begin{enumerate}[{\rm(i)}]
\item $\cb\subset\ub$.
\item $\cb$ is uncountable and has no isolated points.
\item $\dim_H\cb=0$. 
\end{enumerate}
\end{proposition}

\subsection{Main results}
Let $J=[q_L, q_R]$ be a relative plateau with $J\neq[1,M+1]$. Then there is an admissible word $\sv=a_1\ldots a_m$ such that $\al(q_L)=\sv^\f$ and $\al(q_R)=\sv^+(\overline{\sv})^\f$. In particular, $\alpha(q)$ begins with the prefix $\sv^+$ for each $q\in (q_L,q_R]$.
Let
\begin{equation}
\wus_q(J):=\set{(x_i)\in\wus_q: x_1\ldots x_m=\al_1(q)\ldots \al_m(q)=a_1\dots a_m^+}, \qquad q\in (q_L,q_R].
\label{eq:U_q(J)}
\end{equation}
For the special case when $J=J_\emptyset=[1,M+1]$, we set $\wus_q(J):=\wus_q$.
We are now ready to give a characterization of the local dimensional functions $f$, $f_-$ and $f_+$.

\begin{main}\label{main1}\mbox{}

\begin{enumerate}[{\rm(i)}]
\item Let $q\in\overline{\ub}$. Then
\[ f(q)=0\quad\Longleftrightarrow\quad f_-(q)=0\quad\Longleftrightarrow\quad
q\in\cb.\]

\item Let $q\in \overline{\ub} \setminus\cb$. Then  
\[
f_-(q)=\frac{h(\wus_q(J))}{\log q}>0,
\]
where $J=[q_L, q_R]$ is the smallest relative plateau such that  $q\in(q_L, q_R]$. Furthermore,
\[
f_+(q)=\begin{cases}
0 & \textrm{if}\ q\in\overline{\ub}\setminus\ub,\\
\frac{h(\wus_q(J))}{\log q}>0 & \textrm{if}\ q\in\ub \setminus\cb,
\end{cases}
\]
where $J=[q_L, q_R]$ is the smallest relative plateau such that $q\in(q_L, q_R)$. 
As a consequence,
\[
f(q)=\frac{h(\wus_q(J))}{\log q}>0,
\]
where $J=[q_L, q_R]$ is the smallest relative plateau such that $q\in(q_L, q_R)$.
\end{enumerate}
\end{main}

\begin{remark} 
Note the asymmetry between $f_-$ and $f_+$. This is caused by the very different roles played by the left and right endpoints $q_L$ and $q_R$ of a relative plateau. On the one hand, we have $f(q_L)=f_-(q_L)>0$ while $f_+(q_L)=0$. On the other hand, suppose $[q_L,q_R]=J$, and let $I$ be the parent interval of $J$, that is, the relative plateau one level above $J$ that contains $J$. Then
\[
f_-(q_R)=\frac{h(\wus_{q_R}(J))}{\log q_R}>0 \qquad\mbox{and} \qquad f_+(q_R)=\frac{h(\wus_{q_R}(I))}{\log q_R}>0,
\]
and since $\wus_{q_R}(J)\subset \wus_{q_R}(I)$, we have $f_-(q_R)\leq f_+(q_R)$ so $f(q_R)=f_+(q_R)$. In fact, the inequality between $f_-(q_R)$ and $f_+(q_R)$ is almost always strict, with just one possible exception; see Example \ref{ex:endpoints} below for more details.
\end{remark}

Theorem \ref{main1} suggests a closer investigation of the sets $\wus_q(J)$. Our next result gives a detailed description. 

Recall the definition \eqref{eq:q_c} of $q_c(J)$, and let $q_G(J)$ and $q_F(J)$ be the bases in $(q_L, q_R)$ with
\[
\al(q_G(J))=(\sv^+\overline{\sv^+})^\f, \qquad \al(q_F(J))=(\sv^+\overline{\sv}\;\overline{\sv^+}\sv)^\f.
\]
Then $q_G(J)<q_F(J)<q_c(J)$.

\begin{main} \label{main2}
Let $J=[q_L, q_R]$ be a relative plateau generated by the admissible word $\sv$. Then the entropy function 
\[H_J: q\mapsto h(\wus_q(J))\] is a Devil's staircase on $(q_L, q_R]$, i.e., $H_J$ is continuous, non-decreasing and locally constant almost everywhere on $(q_L, q_R]$. 
Furthermore, the set $\wus_q(J)$ has the following structure:
\begin{enumerate}[{\rm(i)}]
\item If $q_L<q\le q_G(J)$, then $\wus_q(J)=\emptyset$.

\item If $q_G(J)<q\le q_F(J)$, then $\wus_q(J)=\big\{\big(\sv^+\overline{\sv^+}\big)^\f\big\}$.

\item If $q_F(J)<q<q_c(J)$, then $\wus_q(J)$ is countably infinite.

\item If $q=q_c(J)$, then $\wus_q(J)$ is uncountable but $H_J(q)=0$.

\item If $q_c(J)<q\le q_R$, then $H_J(q)>0$. 
\end{enumerate}
\end{main}

\begin{remark}
Theorem \ref{main2} can be viewed as a generalization of Theorem \ref{thm:devil-staircase} and the classical result of Glendinning and Sidorov \cite{Glendinning_Sidorov_2001} for the set $\u_q$ with $q\in(1,2]$ and alphabet $\{0,1\}$ (see Proposition \ref{prop:unique expansion-two digits case} below).
\end{remark}

Note that, while the function $H: q\mapsto h(\wus_q)$ is constant on each relative plateau $J$, the set-valued map $F: q\mapsto \wus_q$ is {\em not} constant on $J$. Since $F$ is non-decreasing, it is natural to investigate the variation of the map $q\mapsto \dim_H(\wus_q\setminus\wus_{q_L})$ on $J=[q_L,q_R]$, where the Hausdorff dimension is well defined by equipping the symbolic space $\Omega_M$ with the metric $\rho$ defined by
\begin{equation}
\rho((x_i), (y_i))={2^{-\inf\set{i\ge 0: x_{i+1}\ne y_{i+1}}}}.
\label{eq:rho-metric}
\end{equation}
As an application of Theorem \ref{main2} we have the following.

\begin{corollary} \label{cor:variation-in-plateau}
Let $J=[q_L, q_R]$ be a relative plateau. Then the function
\[
D_J: J\mapsto [0,\infty), \quad q\mapsto \dim_H(\wus_q\setminus\wus_{q_L})
\]
is a Devil's staircase on $J$. Furthermore, $D_J(q)=0$ if and only if $q\le q_c(J)$. 
\end{corollary}

\begin{remark}
Unfortunately, the analogous statement for topological entropy in place of Hausdorff dimension fails: Since sequences in $\wus_q\setminus\wus_{q_L}$ can have arbitrarily long prefixes from any sequence in $\wus_q$, the difference set $\wus_q\setminus\wus_{q_L}$ has the same entropy as $\wus_q$ for all $q\in J\backslash\{q_L\}$.
\end{remark}

Theorems \ref{main1} and \ref{main2} show that the local dimensional functions $f$, $f_-$ and $f_+$ are highly discontinuous on $\overline{\ub}$ (of course, they are everywhere continuous (and equal to zero) on $(1,M+1]\backslash \overline{\ub}$):

\begin{corollary} \label{cor:continuity-of-f}
The local dimensional function $f$ is continuous at $q\in\overline{\ub}$ if and only if $q\in\cb$. The same statement holds for $f_-$ and $f_+$.
\end{corollary}

Next, for any relative entropy plateau $J$ we define the {\em relative bifurcation set}
\[
\bb(J):=\big\{q\in J: h(\wus_p(J))\neq h(\wus_q(J))\ \forall p\in J, p\neq q\big\}.
\]
As a special case, for $J=J_\emptyset=[1, M+1]$ we have $\bb(J)=\bb$.

\begin{main} \label{main-b}
Let $J=J_{\mathbf{i}}=[q_L,q_R]$ be a relative plateau with generating word $\sv=a_1\dots a_m$. Then
\begin{enumerate}[(i)]
\item  
$\bb(J)=\bb(J_{\mathbf i})=J_{\mathbf i}\backslash \bigcup_{j=0}^\f J_{\mathbf{i}j}$;
\item $\bb(J)\subset \ub\cap J$;
\item $\bb(J)$ is Lebesgue null;
\item $\bb(J)$ has full Hausdorff dimension. Precisely,
\[
\dim_H \bb(J)=\dim_H(\ub\cap J)=\frac{\log 2}{m\log q_R};
\]
\item Let $p_0$ be the base with $\alpha(p_0)=\sv^+\overline{\sv}^2\big(\overline{\sv^+}\sv\sv^+\big)^\f$. Then
\[
\dim_H \big((\ub\cap J)\backslash \bb(J)\big)=\frac{\log 2}{3m\log p_0}.
\] 
\end{enumerate}
\end{main}

The representation of $\bb(J)$ in (i) explains why we call the intervals $J_{\mathbf{i}j}$ relative entropy plateaus: They are the maximal intervals on which $h(\wus_q(J_{\mathbf i}))$ is positive and constant.
Comparing statements (i)-(iv) above with the properties of $\bb$ given after \eqref{eq:bifurcation set}, we can say that the set $\bb(J)$ plays the same role on a local level (i.e. within $J$) as the bifurcation set $\bb$ does on a global level. We may observe also that (v) is similar to \cite[Theorem 4]{Allaart-Baker-Kong-17}, which gives the Hausdorff dimension of $\ub\backslash\bb$.

From Proposition \ref{prop:property of C-infity}(i) and Theorem \ref{main-b}(i),(ii) we obtain the following decomposition of $\ub$ into mutually disjoint subsets (recall that $\ub\cap[q_L,q_c(J))=\emptyset$ while $q_c(J)\in\cb$ for any relative plateau $J=[q_L,q_R]$):
\[
\ub=\cb\cup\bb\cup \bigcup_{n=1}^\f\bigcup_{\mathbf i\in\{1,2,\ldots\}^n}\bb(J_{\mathbf i}).
\]

Using Theorems \ref{main1} and \ref{main2} we can answer an open question of Kalle et al.~\cite{Kalle-Kong-Li-Lv-2016}, who asked for the Hausdorff dimension of $\ub\cap[t_1, t_2]$ for any $t_1<t_2$. %Note that $\ub\subset(1, M+1]$. So it suffices to consider the interval $[t_1, t_2]$ contained in $(1, M+1]$. 

\begin{main} \label{main3}
For any $1<t_1<t_2\le M+1$ we have 
\[
\dim_H(\ub\cap[t_1, t_2])=\max\set{\frac{h(\wus_q(J))}{\log q}: q\in\overline{\bb(J)\cap[t_1, t_2]}},
\]
where $J=[q_L, q_R]$ is the smallest relative plateau containing $[t_1, t_2]$. 
\end{main}

\begin{remark}
If $(t_1,t_2)$ intersects the bifurcation set $\bb$, then $J=[1,M+1]$ and the expression in Theorem \ref{main3} simplifies to
\begin{align*}
\dim_H(\ub\cap[t_1, t_2])&=\max\set{\frac{h(\wus_q)}{\log q}: q\in\overline{\bb\cap[t_1, t_2]}}\\
%&=\max\set{\frac{h(\us_q)}{\log q}: q\in\overline{\bb\cap[t_1, t_2]}}\\
&=\max\set{\dim_H \u_q: q\in \overline{\bb\cap[t_1, t_2]}}.
\end{align*}
Setting $t_1=1$ and noting that the map $q\mapsto \dim_H \u_q$ is continuous on $(1,M+1]$ and is decreasing inside each entropy plateau, we obtain Theorem 3 of \cite{Kalle-Kong-Li-Lv-2016}, namely
\[
\dim_H(\ub\cap[1, t])=\max_{q\leq t}\dim_H \u_q, \qquad t\in[1,M+1].
\]
\end{remark}

\subsection{Application to strongly univoque sets}
In 2011, Jordan et al.~\cite{Jordan-Shmerkin-Solomyak-2011} introduced the sets
\begin{equation}
\check\us_q:=\bigcup_{k=1}^\f\set{(x_i)\in\Omega_M: \overline{\al_1(q)\ldots\al_k(q)}\prec x_{n+1}\ldots x_{n+k}\prec \al_1(q)\ldots \al_k(q)~\forall n\ge 0}.
\label{eq:strongly-univoque-set}
\end{equation}
(In fact, their definition was slightly different in that they require the above inequalities only for all sufficiently large $n$. They also defined $\check\us_q$ in a dynamical, rather than a symbolic way, but the definitions are easily seen to be equivalent.) Jordan et al. used the sets $\check\us_q$ to study the multifractal spectrum of Bernoulli convolutions. Recently, the first author \cite{Allaart-2016} used them to characterize the infinite derivatives of certain self-affine functions, and studied them in more detail in \cite{Allaart-2017} where they were called \emph{strongly univoque sets}. 

In view of (\ref{eq:def-widetilde-uq}) it is clear that $\check\us_q\subseteq\wus_q$ for all $q\in(1, M+1]$. On the other hand, $\check\us_q\supset\wus_p$ for all $p<q$  (see \cite{Jordan-Shmerkin-Solomyak-2011} or \cite[Lemma 2.1]{Allaart-2017}). It follows that 
\begin{equation} \label{eq:kkk-1}
\check\us_q=\bigcup_{p<q}\wus_p,
\end{equation}
and, since the function $q\mapsto \dim_H \wus_q$ is continuous, that $\dim_H \check{\us}_q=\dim_H \wus_q$ for every $q$.
A natural question now, is whether $\check\us_q$ could in fact equal $\wus_q$. Following \cite{Allaart-2017}, we define the difference set 
\begin{align}
\begin{split}
\mathbf W_q:&=\wus_q\setminus\check \us_q\\ %=\wus_q\setminus \bigcup_{p<q}\wus_p\\
&=\bigcap_{k=1}^\f \bigcup_{n=0}^\infty\set{(x_i)\in\wus_q: x_{n+1}\ldots x_{n+k}=\al_1(q)\ldots \al_k(q)\textrm{ or }\overline{\al_1(q)\ldots \al_k(q)}},
\end{split}
\label{eq:Wq-def}
\end{align}  
and its projection, $\mathcal W_q:=\pi_q(\mathbf W_q)$. One of the main results in \cite{Allaart-2017} is that $\mathcal W_q\neq\emptyset$ if and only if $q\in\overline{\ub}$, and then $\mathcal W_q$ is in fact uncountable. It is also shown in \cite{Allaart-2017} that $\dim_H\mathcal W_q=0$ whenever $q\in\mathscr C_0$ is a de Vries-Komornik number.

Using the techniques developed in this paper, we can improve on the results of \cite{Allaart-2017} and completely characterize the Hausdorff dimension of $\mathcal W_q$.

\begin{main} \label{main4}
For any {$q\in(1, M+1]$} we have  
\[
\dim_H\mathcal W_q=f_-(q). 
\]
\end{main}

\begin{remark}\mbox{}

\begin{enumerate}
\item By Proposition \ref{prop:local dimension-B} and Theorem \ref{main4} it follows that for each $q\in\bb$ we have 
\[
\dim_H\mathcal W_q=\dim_H\u_q>0.
\]
This provides a negative answer to Question 1.8 of \cite{Allaart-2017}, where it was conjectured that $\dim_H\mathcal W_q<\dim_H\u_q$ for all $q>q_{KL}$. Looking at \eqref{eq:kkk-1}, the above result is not too surprising, since the set-valued function $q\mapsto \wus_q$ is ``most discontinuous" at points of $\bb$.

\item Let $q\in\overline{\ub}$. By Theorem \ref{main1} (i) and Theorem \ref{main4} it follows that $\dim_H\mathcal W_q=0$ if and only if $q\in\cb$. 
This completely characterizes the set $\{q:\dim_H\mathcal W_q=0\}$, extending Theorem 1.5 of \cite{Allaart-2017}. 

\item In view of \eqref{eq:kkk-1} and remark (2) above, we could say that, at points of $\overline{\ub}\backslash \cb$, the set-valued function $q\mapsto \wus_q$ ``jumps" by a set of positive Hausdorff dimension.
\end{enumerate}
\end{remark}

The remainder of this article is organized as follows. In Section \ref{sec:C} we prove Proposition \ref{prop:property of C-infity} and give some examples of points in $\cb_\f$. In Section \ref{sec:map Phi-J} we introduce for each relative plateau $J$ a bijection $\Phi_J$ between symbol spaces and its induced map $\hat{\Phi}_J$ between suitable sets of bases, and develop their properties. These maps allow us to answer questions about relative plateaus and relative bifurcation sets by relating them directly to entropy plateaus $[p_L,p_R]$ and the bifurcation set $\bb$ for the alphabet $\{0,1\}$. This is done in Section \ref{sec: proofs of th-1-2}, where we prove Theorems \ref{main1}, \ref{main2} and \ref{main-b}. Section \ref{sec:proof-of-theorem3} contains a short proof of Theorem \ref{main3}, and Section \ref{sec:proof-of-theorem4} is devoted to the proof of Theorem \ref{main4}.

\section{Properties of the set $\cb$} \label{sec:C}

In this section we prove Proposition \ref{prop:property of C-infity}. Recall that $\alpha(q)$ is the quasi-greedy expansion of $1$ in base $q$. The following useful result is well known (cf.~\cite{Baiocchi_Komornik_2007}).

\begin{lemma} \label{lem:quasi-greedy expansion-alpha-q}
The map $q\mapsto \al(q)$ is strictly increasing and bijective from $(1, M+1]$ to the set of sequences $(a_i)\in\Omega_M$ not ending with $0^\f$ and satisfying
\[
\si^n((a_i))\lle (a_i)\quad\forall n\ge 0.
\]
\end{lemma}

\begin{proof}[Proof of Proposition \ref{prop:property of C-infity}]
(i). It is known that all de Vries-Komornik numbers belong to $\ub$ (cf.~\cite{Kong_Li_2015}), i.e., $\cb_0\subset\ub$. Now let $q\in\cb_\f$ and  $\al(q)=\al_1\al_2\ldots$. Then $q$ belongs to infinitely many relative plateaus. Hence,  there are infinitely many integers $m_1<m_2<\cdots$ such that for each $k$, $\al_1\ldots\al_{m_k}^-$ is admissible, since, if $q$ lies in the relative plateau generated by $b_1\ldots b_n$, then $\al(q)$ must begin with $b_1\ldots b_n^+$. It follows by (\ref{eq:admissible word}) that for each $k$, 
\[
\overline{\al_1\ldots \al_{m_k-i}}\prec \al_{i+1}\ldots \al_{m_k}\lle \al_1\ldots \al_{m_k-i}\quad\forall ~1\le i<m_k.
\]
This implies by induction that $\overline{\al(q)}\prec \si^i(\al(q))\lle\al(q)$ for all $i\in\N$, and hence $q\in\overline{\ub}$ (cf.~\cite{Komornik_Loreti_2007}). But $\overline{\ub}\setminus\ub$ contains only left endpoints of relative plateaus, and these points do not lie in $\cb_\f$. Therefore, $q\in\ub$.

(ii). Clearly, by the construction of $\cb_\f$ it follows that $\cb_\f$ is uncountable, because each relative plateau of level $n$ contains infinitely many pairwise disjoint relative plateaus of level $n+1$. That $\cb$ has no isolated points follows since any right neighborhood of a de Vries-Komornik number contains infinitely many relative plateaus. 

(iii). In \cite{Allaart-Baker-Kong-17}, the following was proved: If $J=[q_L, q_R]$ is a relative plateau generated by $a_1\ldots a_m$, then 
\begin{equation}\label{eq:dimension-local-u}
\dim_H(\overline{\ub}\cap [p, q_R])=\frac{\log 2}{m\log q_R}\quad\textrm{for any }p\in[q_L, q_R).
\end{equation}
(This was stated in \cite{Allaart-Baker-Kong-17} only for entropy plateaus, i.e., the first level relative plateaus, but the proof carries over verbatim to any relative plateau.) %For a detailed explanation, see Lemma \ref{lem:local dimension-relative bifurcation set} below.)

Observe that $\cb_0$ is countable. Furthermore, for a relative plateau $J_{\mathbf i}$ with $\mathbf{i}\in\{1,2,\dots\}^n$, its generating block $a_1\dots a_m$ satisfies $m\geq n$. That $\dim_H\cb=0$ now follows from the definition of $\cb_\f$, the countably stability of Hausdorff dimension, and (\ref{eq:dimension-local-u}). 
\end{proof}

\begin{example}
It is easy to create specific examples of points in $\cb_\f$. For instance, let $\sv=a_1\ldots a_m$ be an admissible word not of the form $\sb\overline{\sb}$ (e.g. $\sv=1110010$ when $M=1$), and construct a sequence $\al_1\al_2\ldots$ as follows: Set $\al_1\ldots\al_m=\sv^+$, and recursively for $k=0,1,\ldots$, let
\[
\al_{3^k m+1}\ldots \al_{2\cdot 3^k m}=\al_{2\cdot 3^k m+1}\ldots \al_{3^{k+1}m}=\overline{\alpha_1\dots\alpha_{3^k m}}^+.
\]
Then $\al_1\al_2\ldots =\al(q)$ for some $q$, and this $q$ lies in $\cb_\f$.

More generally, one can create many more examples by the following procedure. Let again $\sv=a_1\dots a_m$ be any admissible word not of the form $\sb\overline{\sb}$. Now let $\mathbf w$ be a word using the letters $\sv^+, \overline{\sv}, \overline{\sv^+}, \sv$ beginning with $\sv^+$ such that $\mathbf w^-$ is admissible (e.g. $\mathbf w=\sv^+\overline{\sv}^2\overline{\sv^+}\sv^+\overline{\sv^+}\sv^+$). 
%Then $\mathbf w$ corresponds to a finite path in the directed graph in Figure \ref{fig1} with initial label $\sv^+$. 
Put $\mathbf v_0:=\mathbf w$, and recursively, for $k=0,1,2,\dots$, let $\mathbf v_{i+1}$ be the word obtained from $\mathbf v_{i}$ by performing the substitutions
\[
\sv \mapsto \mathbf v_i^-, \qquad \sv^+ \mapsto \mathbf{v}_i, \qquad \overline{\sv}\mapsto \overline{\mathbf{v}_i}^+, \qquad \overline{\sv^+}\mapsto \overline{\mathbf{v}_i}.
\]
Since $\mathbf v_{i+1}$ extends $\mathbf v_i$, the limit $\mathbf{v}:=\lim_{i\to\infty}\mathbf v_i$ exists, and $\mathbf{v}=\alpha(q)$ for some $q$, as the interested reader may check using Lemma \ref{lem:quasi-greedy expansion-alpha-q}. Some reflection reveals that $q\in\mathscr{C}$. The de Vries-Komornik numbers are obtained from $\mathbf w=\sv^+\overline{\sv}$; all other examples obtained this way lie in $\cb_\f$, including the example given at the beginning of this remark, which is obtained from $\mathbf w=\sv^+\overline{\sv}^2$. (For each $i$, $q$ lies in the relative plateau $[q_L(i),q_R(i)]$ given by $\alpha(q_L(i))=(\mathbf{v}_i^-)^\f$ and $\alpha(q_R(i))=\mathbf{v}_i\big(\overline{\mathbf{v}_i}^+\big)^\f$; we leave the details for the interested reader.)
\end{example}

\section{Descriptions of the map $\Phi_J$ and the induced map $\hat\Phi_J$} \label{sec:map Phi-J}

In this section we fix a relative plateau $J=[q_L, q_R]$ with
\[
\al(q_L)=\sv^\f\quad\textrm{and}\quad \al(q_R)=\sv^+(\overline{\sv})^\f
\]
for some admissible word $\sv=a_1\ldots a_m$. Note by Definition \ref{def:admissible-words} and Lemma \ref{lem:quasi-greedy expansion-alpha-q} that $q_L$ and $q_R$ are well defined and $q_L<q_R$.

A special role in this paper is played by sets associated with the alphabet $\{0,1\}$. When the alphabet $\{0,1\}$ is intended, we will {affix} a superscript $^*$ to our notation. Thus, $\bb^*=\bb$ when $M=1$, $\ub^*=\ub$ when $M=1$, etc. We call $\bb^*$ the {\em reference bifurcation set}. The key to the proofs of our main results, and the main methodological innovation of this paper, is the construction of a bijection $\hat{\Phi}_J$ from $\bb(J)$ to $\bb^*$. More generally, $\hat{\Phi}_J$ maps important points of $J$ to important points of $(1,2]$ for the case $M=1$. Associated with $\hat{\Phi}_J$ is a symbolic map $\Phi_J$ which maps each set $\wus_q(J)$ to the symbolic univoque set $\wus_{\hat{q}}^*$ for $M=1$, where $\hat{q}=\hat{\Phi}_J(q)$. By using properties of the maps $\Phi_J$ and $\hat{\Phi}_J$, many classical results on univoque sets with alphabet $\{0,1\}$ can be transferred to the relative entropy plateaus and the sets $\wus_q(J)$.

Figure \ref{fig1} shows a directed graph $G$ with two sets of labels. The labeled graph $\mathcal G=(G, \mathcal L)$ with labels in $\mathcal L:=\big\{\sv, \sv^+, \overline{\sv}, \overline{\sv^+}\big\}$ is right-resolving, i.e. the out-going edges from the same vertex in $\mathcal G$ have different labels. Let $X(J)$ be the set of infinite sequences determined by the automata $\mathcal G=(G, \mathcal L)$, beginning at the ``Start" vertex (cf.~\cite{Lind_Marcus_1995}).  We emphasize that each digit $\sd$ in $\mathcal L$ is a block of length $m$, and any sequence in $X(J)$ is an infinite concatenation of blocks from $\mathcal{L}$.

Likewise, the {\em reference labeled graph} $\mathcal G^*=(G, \mathcal L^*)$ with labels in $\mathcal L^*:=\set{0,1}$ is right-resolving. Hence for each $q\in(1,2]$ {the quasi-greedy expansion $\de(q)$ of $1$ in base $q$} is uniquely represented by an infinite path determined by the automata $\mathcal G^*$. Let $X^*\subset \{0,1\}^\N$ be the set of all infinite sequences determined by the automata $\mathcal G^*$, and note that $X^*=\{(x_i)\in\{0,1\}^\N:x_1=1\}$.
Then $\set{\de(q): q\in(1, 2]}\subset X^*$, the inclusion being proper in view of Lemma \ref{lem:quasi-greedy expansion-alpha-q}.

\begin{figure}[h!]
  \centering
  \begin{tikzpicture}[->,>=stealth',shorten >=1pt,auto,node distance=4cm,
                    semithick]

  \tikzstyle{every state}=[minimum size=0pt,fill=none,draw=black,text=black]

  \node[state] (A)                    { $A$};
  \node[state]         (B) [ right of=A] {$B$ };
  \node[state]         (C) [ above of=A] {$Start$};

  \path[->,every loop/.style={min distance=0mm, looseness=40}]
  (C) edge[->] node{$\sv^+\; / \; 1$} (A)
  
   (A) edge [loop left,->]  node {$\overline{\sv}\;/\; 1$} (A)
            edge  [bend left]   node {$\overline{\sv^+}\;/\; 0$} (B)

        (B) edge [loop right] node {$\sv\; /\; 0$} (B)
            edge  [bend left]            node {$\sv^+\; / \; 1$} (A);
\end{tikzpicture}

\caption{The labeled graph $\mathcal G=(G, \mathcal L)$ with labels $\mathcal L=\big\{\sv, \sv^+, \overline{\sv}, \overline{\sv^+}\big\}$, and the  reference labeled graph $\mathcal G^*=(G, \mathcal L^*)$ with labels $\mathcal L^*=\set{0, 1}$. The map $\phi: \mathcal L\ra \mathcal L^*$ is defined by $
\phi(\overline{\sv^+})=\phi(\sv)=0 $ and $ \phi(\sv^+)=\phi(\overline{\sv})=1.
$
}
\label{fig1}
\end{figure}

\begin{proposition}\label{prop:property of Uq(J)}
$\wus_q(J)\subset X(J)$ for {every} $q\in(q_L, q_R]$.
\end{proposition}

To prove the proposition we need the following. 

\begin{lemma}\label{lem:uq-xb}
Any sequence $(x_i)\in\Omega_M$ satisfying $x_1\ldots x_m=\sv^+$ and 
\begin{equation}\label{eq:inequality-1}
\overline{\sv^+}\sv^\f \lle \si^n((x_i))\lle \sv^+(\overline{\sv})^\f\quad\forall ~ n\ge 0
\end{equation}
belongs to $X(J)$. 
\end{lemma}

\begin{proof}
Take a sequence $(x_i)$ satisfying $x_1\ldots x_m=\sv^+=a_1\ldots a_m^+$ and (\ref{eq:inequality-1}).
Then by (\ref{eq:inequality-1}) with $n=0$ and $n=m$ it follows that 
\[
\overline{\sv^+}\lle x_{m+1}\ldots x_{2m}\lle \overline{\sv}.
\]
So,  either $x_{m+1}\ldots x_{2m}=\overline{\sv^+}$ or $x_{m+1}\ldots x_{2m}=\overline{\sv}$.

\begin{enumerate}[{\rm(i)}]
\item If $x_{m+1}\ldots x_{2m}=\overline{\sv^+}$, then by (\ref{eq:inequality-1}) with $n=m$ and $n=2m$ it follows that the next block $x_{2m+1}\ldots x_{3m}=\sv^+$ or $\sv$.

\item If $x_{m+1}\ldots x_{2m}=\overline{\sv}$, then $x_{1}\ldots x_{2m}=\sv^+\overline{\sv}$. By (\ref{eq:inequality-1}) with $n=0$ and $n=2m$ it follows that the next block $x_{2m+1}\ldots x_{3m}=\overline{\sv^+}$ or $\overline{\sv}$.
\end{enumerate}

Iterating the above reasoning and referring to Figure \ref{fig1} we conclude that $(x_i)\in X(J)$.  
\end{proof}

\begin{proof}[Proof of Proposition \ref{prop:property of Uq(J)}]
Take $q\in(q_L, q_R]$.  {Since $\al(q_L)=\sv^\f$ and $\al(q_R)=\sv^+(\overline{\sv})^\f$, Lemma \ref{lem:quasi-greedy expansion-alpha-q} implies that $\al_1(q)\ldots \al_m(q)=\sv^+$ and $\al(q)\lle \al(q_R)=\sv^+(\overline{\sv})^\f$. Hence, by \eqref{eq:U_q(J)}, \eqref{eq:def-widetilde-uq} and Lemma \ref{lem:uq-xb}, it follows that $\wus_q(J)\subset X(J)$.}
\end{proof}

We next introduce the right bifurcation set $\vb$ for the set-valued map $q\mapsto \wus_q$ (cf.~\cite{DeVries_Komornik_2008}):
\[
\vb:=\set{q\in(1, M+1]: \wus_r\ne \wus_q\ \forall r>q}.
\]
Recall that $\ub$ is the set of univoque bases. The following characterizations of $\ub$ and $\vb$ are proved in \cite{Vries-Komornik-Loreti-2016}.

\begin{lemma}\label{lem:characterization of V-U}\mbox{}

\begin{enumerate}[{\rm(i)}]  
\item $q\in\ub\setminus\set{M+1}$ if and only if
$\overline{\al(q)}\prec \si^n(\al(q))\prec \al(q)$ {for all } $n\ge 1$.
\item $q\in\vb$ if and only if 
$\overline{\al(q)}\lle \si^n(\al(q))\lle \al(q)$ {for all } $n\ge 1$. %That is, $q\in\vb$ if and only if $\alpha(q)\in\vs$.
\end{enumerate}
\end{lemma}

Clearly, Lemma \ref{lem:characterization of V-U} implies that $\ub\subset\vb$. Furthermore, $\vb\setminus\ub$ is at most countable.  Set
\begin{align*}
 \us(J):=\set{\al(q): q\in\ub\cap (q_L, q_R]}\quad\textrm{and}\quad
\vs(J):=\set{\al(q): q\in\vb\cap (q_L, q_R]}.
\end{align*}
Then $\us(J)\subset\vs(J)$. %Since $q_L\notin\ub$, we can also write $\us(J)=\set{\al(q): q\in\ub\cap J}$. 
As a consequence of Lemmas \ref{lem:uq-xb} and \ref{lem:characterization of V-U} we have the following.

\begin{proposition} \label{proposition:U(J)-V(J)}
$\us(J)\subset\vs(J)\subset X(J)$. Furthermore, 
\begin{gather*}
\us(J)=\big\{(\sc_i)\in X(J): \overline{(\sc_i)}\prec \si^n((\sc_i))\prec (\sc_i)~\forall n\ge 1\big\},\\
\vs(J)=\big\{(\sc_i)\in X(J): \overline{(\sc_i)}\lle \si^n((\sc_i))\lle (\sc_i)~\forall n\ge 1\big\}. 
\end{gather*}
\end{proposition}

We shall also need the following sets. For $M=1$ we denote by
\begin{align*}
&\ub^*:=\ub\quad\textrm{and}\quad \us^*:=\set{\al^*(q): q\in\ub^*},\\
&\vb^*:=\vb\quad\textrm{and}\quad \vs^*:=\set{\al^*(q): q\in\vb^*}.
\end{align*}
Then by Lemma \ref{lem:characterization of V-U} with $M=1$ it follows that  
\begin{equation}\label{eq:U-star-V-star}
\begin{split}
\us^*\setminus\set{1^\f}&=\left\{(a_i)\in\set{0,1}^\N: (1-a_i)\prec \si^n((a_i))\prec (a_i)~\forall n\ge 1\right\},\\
\vs^*&=\set{(a_i)\in\set{0,1}^\N: (1-a_i)\lle \si^n((a_i))\lle (a_i)~\forall n\ge 1}.
\end{split}
\end{equation}

\subsection{Description of $\Phi_J$}
We now define a map $\phi: \mathcal{L}\to\mathcal{L}^*$ by 
\begin{equation}
\phi(\overline{\sv^+})=\phi(\sv)=0, \qquad\mbox{and} \qquad \phi(\overline{\sv})=\phi(\sv^+)=1.
\label{eq:phi}
\end{equation}
Then $\phi$ induces a block map $\Phi_J: X(J)\ra X^*$ defined by 
\[
\Phi_J((\sd_i)):=\phi(\sd_1)\phi(\sd_2)\ldots.
\]
  
\begin{proposition} \label{prop:chareacterization of Phi-J}
The map $\Phi_J: X(J)\ra X^*$ is strictly increasing and bijective. Furthermore, 
\[
\Phi_J(\us(J))=\us^*\quad\textrm{and}\quad \Phi_J(\vs(J))=\vs^*.
\]
\end{proposition}

First we verify that $\Phi_J$ is a bijection.

\begin{lemma} \label{lem:bijective map-Phi}
The map $\Phi_J: X(J)\ra X^*$ is strictly increasing and bijective. 
\end{lemma}
	
\begin{proof}
Note by Definition \ref{def:admissible-words} that the blocks in $\mathcal{L}$ are ordered by $\overline{\sv^+}\prec\overline{\sv}\prec\sv\prec\sv^+$.
Take two sequences $(\sc_i), (\sd_i)\in X(J)$ with $(\sc_i)\prec (\sd_i)$. Then $\sc_1=\sd_1=\sv^+$, and there is an integer $k\ge 2$ such that $\sc_1\ldots \sc_{k-1}=\sd_1\ldots \sd_{k-1}$ and $\sc_k\prec \sd_k$.  We will show that $\phi(\sc_k)<\phi(\sd_k)$. To this end we consider two cases (see Figure \ref{fig1}):
\begin{itemize}
\item[(I).] If $\sc_{k-1}=\sv^+$ or $\overline{\sv}$, then $\sc_k=\overline{\sv^+}$ and $\sd_k=\overline{\sv}$, and so $\phi(\sc_k)=0$ and $\phi(\sd_k)=1$.

\item[(II).] If $\sc_{k-1}=\sv$ or $\overline{\sv^+}$, then $\sc_k=\sv$ and $\sd_k=\sv^+$, so again $\phi(\sc_k)=0$ and $\phi(\sd_k)=1$. 
\end{itemize} 
Thus, $\Phi_J$ is strictly increasing on $X(J)$. Finally, since the labeled graphs $\mathcal G$ and $\mathcal G^*$ are both right-resolving, the definitions of $X(J)$ and $X^*$ imply that $\Phi_J$ is bijective. 
\end{proof}

\begin{lemma} \label{lem:description-Phi-J}
The following statements are equivalent for sequences $(\sc_i), (\sd_i)\in X(J)$. 
\begin{enumerate}[{\rm(i)}]
\item %$(\sc_i), (\sd_i)\in X(J)$ satisfy
$
\overline{(\sd_i)}\prec\si^n((\sc_i))\prec (\sd_i)~~\forall n\ge 0.
$
\item %$(\sc_i), (\sd_i)\in X(J)$ satisfy
$
\overline{(\sd_i)}\prec\si^{mn}((\sc_i))\prec (\sd_i)~~\forall n\ge 0.
$
\item 
The image sequences $(x_i):=\Phi_J((\sc_i)), (y_i):=\Phi_J((\sd_i))$ in $X^*$ satisfy
\[
{(1-y_i)}\prec \si^n((x_i))\prec (y_i)~~\forall n\ge 0.
\]
\end{enumerate}
\end{lemma}

\begin{proof}
Since $\sv=a_1\ldots a_m$ is admissible, Definition \ref{def:admissible-words} implies
\begin{equation*}
\overline{a_1\ldots a_m^+}\prec a_{i+1}\ldots a_m^+\overline{a_1\ldots a_i}\prec a_1\ldots a_m^+
\end{equation*}
and
\begin{equation*}
\overline{a_1\ldots a_m^+}\prec a_{i+1}\ldots a_m a_1\ldots a_i\prec a_1\ldots a_m^+
\end{equation*}
for all $1\le i<m$. Using $\sd_1=\sv^+=a_1\dots a_m^+$ and $(\sc_i)\in X(J)$ this proves the equivalence (i) $\Leftrightarrow$ (ii).

Next, we prove (ii) $\Rightarrow$ (iii).  We only verify the second inequality in (iii); the first one can be proved in the same way. Take $(\sc_i), (\sd_i)\in X(J)$ satisfying the inequalities in (ii), and let $(x_i):=\Phi_J((\sc_i))$, $(y_i):=\Phi_J((\sd_i))$. 
Fix $n\geq 0$. If $\sc_{n+1}\in\{\overline{\sv^+}, \sv\}$, then ${x}_{n+1}=\phi(\sc_{n+1})=0<1=\phi(\sd_1)=y_1$, using that $\sd_1=\sv^+$. Furthermore, if $\sc_{n+1}=\sv^+$, then $\sigma^{mn}((\sc_i))\in X(J)$ and the second inequality in (iii) follows from Lemma \ref{lem:bijective map-Phi}. Therefore, the critical case is when $\sc_{n+1}=\overline{\sv}$, which we assume for the remainder of the proof.

{Since $(\sc_i)\in X(J)$,  there is $0\le j<n$ such that $\sc_{j+1}\ldots \sc_{n+1}=\sv^+(\overline{\sv})^{n-j}$. (See Figure \ref{fig1}.) Furthermore, since $\sc_{j+1}\sc_{j+2}\ldots \prec \sd_1\sd_2\ldots\lle \sv^+(\overline{\sv})^\f$, there is a number $k\geq n-j$ such that
\begin{equation}
\sc_{j+1}\dots \sc_{j+k+2}=\sv^+(\overline{\sv})^k\overline{\sv^+}, \qquad\mbox{and} \qquad \sd_1\dots \sd_{k+1}=\sv^+(\overline{\sv})^k.
\label{c-and-d}
\end{equation}
The second equality in \eqref{c-and-d} yields $y_1\dots y_{k+1}=\phi(\sd_1)\ldots\phi(\sd_{k+1})=1^{k+1}$, and the first equality implies 
\[
\sc_{n+1}\dots \sc_{j+k+2}=(\overline{\sv})^{k-(n-j)+1}\overline{\sv^+}.
\]
Hence,
\[
x_{n+1}\dots x_{j+k+2}=\phi(\sc_{n+1})\ldots \phi(\sc_{j+k+2})=1^{k-(n-j)+1}0\prec 1^{k-(n-j)+2}=y_1\dots y_{k-(n-j)+2},
\]
since $j<n$ implies $k-(n-j)+2\leq k+1$. Therefore, $\sigma^n((x_i))\prec (y_i)$, which gives the second inequality in (iii).
}

Finally, we prove (iii) $\Rightarrow$ (ii). First we verify the second inequality of (ii). Let $({x}_i), (y_i)\in X^*$ and let $(\sc_i), (\sd_i)\in X(J)$ such that $\Phi_J((\sc_i))=({x}_i), \Phi_J((\sd_i))=(y_i)$. Fix $n\ge 0$. We may assume $\sc_{n+1}=\sv^+$, as otherwise the inequality is trivial. But then $\sc_{n+1}\sc_{n+2}\ldots \in X(J)$, and since ${x}_{n+1}{x}_{n+2}\ldots\prec y_1 y_2\ldots$ it follows from Lemma \ref{lem:bijective map-Phi} that $\sc_{n+1}\sc_{n+2}\ldots \prec \sd_1\sd_2\ldots$. This proves the second inequality of (ii). The first inequality is verified analogously. 
\end{proof}

\begin{proof}[Proof of Proposition \ref{prop:chareacterization of Phi-J}]
In view of Lemma \ref{lem:bijective map-Phi} it remains to prove 
\[\Phi_J(\us(J))=\us^*\quad \textrm{and}\quad \Phi_J(\vs(J))=\vs^*.\]
Since the proof of the second equality is similar, we only prove the first one. 

Let $(\sc_i)\in\us(J)$, and $(x_i):=\Phi_J((\sc_i))$. Then by Proposition \ref{proposition:U(J)-V(J)} it follows that 
\[
\sc_1=\sv^+,\quad\textrm{and}\quad \overline{(\sc_i)}\prec \si^n((\sc_i))\prec (\sc_i)~\forall n\ge 1.
\] 
By Lemma \ref{lem:description-Phi-J} with $(\sc_i)=(\sd_i)$ this is equivalent to 
\[
x_1=1,\quad\textrm{and}\quad (1-x_i)\prec \si^n((x_i))\prec (x_i)~\forall n\ge 1.
\] 
So, by (\ref{eq:U-star-V-star}) we have $(x_i)\in\us^*$, and thus $\Phi_J(\us(J))\subseteq \us^*$.

Conversely, take $(x_i)\in\us^*\subset X^*$. By Lemma \ref{lem:bijective map-Phi} there exists a (unique) sequence $(\sc_i)\in X(J)$ such that $\Phi_J((\sc_i))=(x_i)$. If $(x_i)=1^\f$, then $(\sc_i)=\sv^+(\overline{\sv})^\f=\al(q_R)\in\us(J)$. If $(x_i)\in\us^*\setminus\set{1^\f}$, then by (\ref{eq:U-star-V-star}), Lemma  \ref{lem:description-Phi-J}  and the same argument as above it follows that $(\sc_i)\in\us(J)$. Hence, $\Phi_J(\us(J))=\us^*$.
\end{proof}

 \begin{figure}[h!]
  \centering
\usetikzlibrary{matrix}
 
\begin{tikzpicture}[scale=2cm]
  \matrix (m) [matrix of math nodes,row sep=3em,column sep=4em,minimum width=2em]
  {
    \vb\cap(q_L, q_R] & \vb^* \\
     \vs(J) & \vs^* \\};
  \path[-stealth]
    (m-1-1) edge node [left] {$\al$} (m-2-1)
            edge   node [above] {$\hat\Phi_J$} (m-1-2)
    (m-2-1.east|-m-2-2) edge  node [above] {$\Phi_J$} (m-2-2)
    (m-2-2) edge node [right] {$(\al^*)^{-1}$} (m-1-2);
\end{tikzpicture}
\caption{The figure of the exchange map between $\hat\Phi_J$ and $\Phi_J$.}\label{figure:2}.
\end{figure}

\subsection{Description of the induced map $\hat\Phi_J$}
Recall from Proposition \ref{proposition:U(J)-V(J)} that $\vs(J)\subset X(J)$. Hence Proposition \ref{prop:chareacterization of Phi-J} implies that the bijective map $\Phi_J: \vs(J)\ra \vs^*$ induces an increasing bijective map (see Figure \ref{figure:2})
\[
\hat\Phi_J: \vb\cap(q_L, q_R]\ra \vb^*;\qquad q\mapsto (\al^*)^{-1}\circ\Phi_J\circ\al(q).
\]

The relevance of the map $\hat{\Phi}_J$ is made clear by the following proposition. Here, for $M=1$ and $q\in(1,2]$ we write $\wus_q^*:=\wus_q$. 

\begin{proposition} \label{th:characterization of hat-Phi-J} \mbox{}

\begin{enumerate}
[{\rm(i)}]

\item $\hat\Phi_J: \vb\cap(q_L, q_R]\ra \vb^*$ is a strictly increasing homeomorphism.

\item $\hat\Phi_J(\ub\cap(q_L, q_R])=\hat\Phi_J(\ub\cap J)=\ub^*$. %and $\hat\Phi_J(\vb\cap(q_L, q_R])=\vb^*$.

%\item $\hat\Phi_J(\bb(J))=\bb^*$.

\item For any $q\in\vb\cap(q_L, q_R]$ and $\hat q:=\hat\Phi_J(q)$ we have
\[
\Phi_J\big(\wus_q(J)\big)=\big\{(x_i)\in\wus_{\hat q}^*: x_1=1\big\} \qquad \textrm{and} \qquad  h(\wus_q(J))=\frac{h(\wus_{\hat q}^*)}{m}.
\]
\end{enumerate}
\end{proposition}

\begin{remark}
In the special case when $M=1$, Proposition \ref{th:characterization of hat-Phi-J}(ii) implies that $\ub$ can be viewed as an attractor of an inhomogeneous infinite iterated function system: Since $\ub^*=\ub$ in this case, we can write
\[
\ub=\bigcup_{i=1}^\infty \hat{\Phi}_{J_i}^{-1}(\ub)\cup \big(\bb\cup \{q_{KL}\}\big),
\]
using \eqref{eq:relation-entropy plateaus-bifurcation set} and the definition of $J_i$.
\end{remark}

Part (i) of Proposition \ref{th:characterization of hat-Phi-J} follows from the following lemma, which proves something stronger: it implies H\"older properties of the maps $\hat{\Phi}_J$ and $\hat\Phi_J^{-1}$. These will be important later for Hausdorff dimension calculations.
%First, we prove the continuity of $\hat\Phi_J$ and its inverse. 

\begin{lemma} \label{lem:Holder continuity of hatPhi}
There exist constants $c_1, c_2>0$ such that for any $q_1, q_2\in\vb\cap(q_L, q_R]$ with $q_1<q_2$ we have   
\begin{equation}\label{eq:continuity-hat-Phi}
c_1(q_2-q_1)^{\frac{\log \hat q_2}{m\log q_2}}\le \hat\Phi_J(q_2)-\hat\Phi_J(q_1) \le c_2 (q_2-q_1)^{\frac{\log \hat q_2}{m\log q_2}},
\end{equation}
where $\hat q_i:=\hat\Phi_J(q_i)$ for $i=1, 2$.  
\end{lemma}

\begin{proof}
We only demonstrate the second inequality of (\ref{eq:continuity-hat-Phi}), since the proof of the first inequality is very similar.  

Let $q_1, q_2\in\vb\cap(q_L, q_R]$ with $q_1<q_2$, and let $\hat q_i:=\hat\Phi_J(q_i)$, $i=1,2$. Then $\hat q_1<\hat q_2$ by the monotonicity of $\hat\Phi_J$. Furthermore, Lemma \ref{lem:quasi-greedy expansion-alpha-q} gives $\al(q_1)\prec \al(q_2)$.
Note by Proposition \ref{proposition:U(J)-V(J)} that $\al(q_1), \al(q_2)\in\vs(J)\subset X(J)$. Therefore, $\al(q_1), \al(q_2)$ can be written as $\al(q_1)=(\sc_i)$ and $\al(q_2)=(\sd_i)$ with $\sc_i, \sd_i\in\big\{\sv, \sv^+, \overline{\sv}, \overline{\sv^+}\big\}$ for all $i\ge 1$. In view of Figure \ref{fig1}, there exists $n\ge 2$ such that 
\begin{equation}\label{eq:kd-1}
\sc_1\ldots \sc_{n-1}=\sd_1\ldots\sd_{n-1}\quad\textrm{and}\quad \sc_n\prec \sd_n.
\end{equation}
Observe that $\al(q_2)\in\vs(J)$. By Proposition \ref{proposition:U(J)-V(J)} it follows that 
\[
\si^{mn}(\al(q_2))\lge \overline{\al(q_2)}\lge \overline{\sv^+}\sv^\f\lge 0^m 10^\f,
\]
which implies
\[
1=\sum_{i=1}^\f\frac{\al_i(q_2)}{q_2^i}\ge \sum_{i=1}^{mn}\frac{\al_i(q_2)}{q_2^i}+\frac{1}{q_2^{mn+m+1}}.
\]
Therefore, by (\ref{eq:kd-1}) with $\al(q_1)=(\sc_i)$ and  $\al(q_2)=(\sd_i)$ it follows that 
\begin{align*}
\frac{1}{q_2^{mn+m+1}}&\le 1-\sum_{i=1}^{mn}\frac{\al_i(q_2)}{q_2^i}=\sum_{i=1}^{\f}\frac{\al_i(q_1)}{q_1^i}-\sum_{i=1}^{mn}\frac{\al_i(q_2)}{q_2^i}\\
&\le \sum_{i=1}^{mn}\left(\frac{\al_i(q_2)}{q_1^i}-\frac{\al_i(q_2)}{q_2^i}\right)\\
&\le \sum_{i=1}^\f\left(\frac{M}{q_1^i}-\frac{M}{q_2^i}\right)=\frac{M(q_2-q_1)}{(q_1-1)(q_2-1)}.
\end{align*}
Since $q_L<q_1<q_2<q_R$, we obtain
\begin{equation}\label{eq:kd-2}
\frac{1}{q_2^{mn}}\le \frac{M q_R^{m+1}}{(q_L-1)^2}(q_2-q_1).
\end{equation} 

Write {$({x}_i):=\Phi_J((\sc_i))$ and $(y_i):=\Phi_J((\sd_i))$.} Then (\ref{eq:kd-1}) and {Lemma \ref{lem:bijective map-Phi}} imply
\begin{equation}\label{eq:kd-3}
{x}_1\ldots {x}_{n-1}=y_1\ldots y_{n-1}\quad\textrm{and}\quad {x}_n<y_n.
\end{equation}
Note that $({x}_i), (y_i)\in\vs^*$. By the definition of $\hat\Phi_J$ we have 
{$({x}_i)=\Phi_J(\al(q_1))=\al^*\big(\hat{\Phi}_J(q_1)\big)=\al^*(\hat q_1)$, and similarly $(y_i)=\al^*(\hat q_2)$.} So, by (\ref{eq:kd-3}) it follows that 
\begin{align*}
\hat\Phi_J(q_2)-\hat\Phi_J(q_1)=\hat q_2-\hat q_1
&=\sum_{i=1}^\f\frac{y_i}{\hat q_2^{i-1}}-\sum_{i=1}^\f\frac{{x}_i}{\hat q_1^{i-1}}\\
&\le \sum_{i=1}^{n-1}\left(\frac{y_i}{\hat q_2^{i-1}}-\frac{{x}_i}{\hat q_1^{i-1}}\right)+\sum_{i=n}^\f\frac{y_i}{\hat q_2^{i-1}}\\
&\le \frac{1}{\hat q_2^{n-2}}\le \frac{4}{\hat q_2^n}.
\end{align*}
Here the second inequality follows from the definition of the quasi-greedy expansion $\al^*(\hat q_2)=(y_i)$.
This, together with (\ref{eq:kd-2}), yields
\[
\hat\Phi_J(q_2)-\hat\Phi_J(q_1) \le 4\left(\frac{1}{q_2^{mn}}\right)^{\frac{\log \hat q_2}{m\log q_2}}\le c_2 (q_2-q_1)^{\frac{\log \hat q_2}{m\log q_2}}
\] 
for some constant $c_2$ independent of $q_1$ and $q_2$. 
\end{proof}

\begin{proof}[Proof of Proposition \ref{th:characterization of hat-Phi-J}]
That $\hat{\Phi}_J$ is increasing and bijective follows since it is the composition of increasing and bijective maps.
By Lemma \ref{lem:Holder continuity of hatPhi}, $\hat{\Phi}_J$ and $\hat{\Phi}_J^{-1}$ are continuous. Thus, we have proved (i). 
Since $q_L\not\in\ub$, we have $\us(J)=\{\alpha(q): q\in \ub\cap J\}$. Thus, statement (ii) is a direct consequence of Proposition \ref{prop:chareacterization of Phi-J}. It remains only to establish (iii).

Take $q\in\vb\cap(q_L, q_R]$. Then by Proposition \ref{proposition:U(J)-V(J)} we have $\al(q)\in \vs(J)\subset X(J)$. Note by Proposition \ref{prop:property of Uq(J)} that $\wus_q(J)\subset X(J)$. Now take a sequence $(\sc_i)\in X(J)$ and let $(x_i):=\Phi_J((\sc_i))\in X^*$. Then we have the equivalences
\begin{align*}
(\sc_i)\in\wus_q(J) \quad &\Longleftrightarrow \quad \sc_1=\sv^+ \quad\mbox{and} \quad
\overline{\al(q)}\prec \si^n((\sc_i))\prec \al(q)\quad\forall n\ge 0 \\
&\Longleftrightarrow \quad x_1=1 \quad\ \ \mbox{and} \quad 
\Phi_J\left(\overline{\al(q)}\right)\prec \si^n(({x}_i))\prec \Phi_J(\al(q))\quad\forall n\ge 0 \\
&\Longleftrightarrow \quad x_1=1 \quad\ \ \mbox{and} \quad
(1-\al_i^*(\hat q))\prec \si^n(({x}_i))\prec \al^*(\hat q)\quad\forall n\ge 0 \\
&\Longleftrightarrow \quad x_1=1 \quad\ \ \mbox{and} \quad (x_i)\in \wus^*_{\hat q}
\end{align*}
Here the second equivalence follows by Lemma \ref{lem:description-Phi-J}  {with $(\sd_i)=\alpha(q)$}, and the third equivalence follows since $\al^*(\hat q)=\Phi_J(\al(q))$.
As a result, $\Phi_J(\wus_q(J))=\big\{(x_i)\in\wus^*_{\hat q}: x_1=1\big\}$.

For the entropy statement we observe that the map 
\[
\Phi_J: \wus_q(J) \ra \wus^*_{\hat q}(1):=\big\{(x_i)\in\wus^*_{\hat q}: x_1=1\big\};\qquad (\sc_i)\mapsto (\phi(\sc_i))
\] 
is a bijective {$m$-block map}. Furthermore, $\wus^*_{\hat q}$ is the disjoint union of $\wus^*_{\hat q}(1)$ with its reflection $\big\{(1-x_i): (x_i)\in\wus_{\hat q}^*(1)\big\}$. This implies $h(\wus_q(J))=h(\wus^*_{\hat q})/m$.
\end{proof}

\section{Proofs of Theorems \ref{main1}, \ref{main2} and \ref{main-b}} \label{sec: proofs of th-1-2}

Our first goal is to prove Theorem \ref{main1}. We begin with a useful lemma.

\begin{lemma} \label{lem:dense-plateaus}
Let $J=J_{\mathbf i}=[q_L,q_R]$ be a relative entropy plateau. Then the union $\bigcup_{j=1}^\infty J_{\mathbf{i}j}$ is dense in $(q_c(J),q_R]$.
\end{lemma}

\begin{proof}
Recall from \cite{AlcarazBarrera-Baker-Kong-2016} that the entropy plateaus $J_j^*$, $j\in\N$ are dense in $(q_{KL}^*,2]$. Note that we may order the intervals $J_{\mathbf{i}j}$, $j\in\N$ so that {$\hat{\Phi}_J(\vb\cap J_{\mathbf{i}j})=\vb^*\cap J_j^*$} for each $j$.
Hence, the result follows from the continuity of $\hat{\Phi}_J^{-1}$ (cf. Lemma \ref{lem:Holder continuity of hatPhi}).
\end{proof}

For $M=1$ and $q\in(1,2]$ we denote the left and right local dimensional functions by
\[
f_-^*(q):=\lim_{\delta\ra 0}\dim_H(\ub^*\cap(q-\delta, q)) \qquad \mbox{and} \qquad
f_+^*(q):=\lim_{\delta\ra 0}\dim_H(\ub^*\cap(q, q+\delta)),
\]
respectively.

\begin{lemma} \label{lem:f-minus-bridge}
Let $J=[q_L,q_R]$ be a relative plateau generated by a word $a_1\dots a_m$, and $q\in\overline{\ub}\cap(q_L,q_R]$. Then
\[
f_-(q)=\frac{\log \hat{q}}{m\log q}f_-^*(\hat{q}),
\]
where $\hat{q}:=\hat{\Phi}_J(q)$.
\end{lemma}

\begin{proof}
By the assumption on $q$, we have that $q\in\vb$ and there is a sequence $(p_i)$ in $\vb\cap J$ such that $p_i<q$ for each $i$, and $p_i\nearrow q$ (cf. \cite{Vries-Komornik-Loreti-2016}). 
Let $\hat{p}_i:=\hat{\Phi}_J(p_i)$; then $\hat{p}_i<\hat{q}$ for each $i$, and $\hat{p}_i\nearrow \hat{q}$. 

Observe from Lemma \ref{lem:Holder continuity of hatPhi} that for each $i$, $\hat{\Phi}_J$ is H\"older continuous on $[p_i,q]$ with exponent $\log\hat{p}_i/(m\log q)$, and $\hat{\Phi}_J^{-1}$ is H\"older continuous on $[\hat{p}_i,\hat{q}]$ with exponent $m\log p_i/\log\hat{q}$. It follows on the one hand that
\begin{equation*}
\dim_H(\ub^*\cap(\hat{p}_i,\hat{q}))=\dim_H \hat{\Phi}_J(\ub\cap(p_i,q))
\leq \frac{m\log q}{\log \hat{p}_i}\dim_H(\ub\cap(p_i,q)),
\end{equation*}
so letting $i\to\infty$ we obtain
\[
f_-^*(\hat{q})\leq \frac{m\log q}{\log\hat{q}}f_-(q).
\]
On the other hand, 
\[
\dim_H(\ub\cap(p_i,q))=\dim_H \hat{\Phi}_J^{-1}\big(\ub^*\cap(\hat{p}_i,\hat{q})\big)
\leq \frac{\log\hat{q}}{m\log p_i}\dim_H\big(\ub^*\cap(\hat{p}_i,\hat{q})\big),
\]
so letting $i\to\infty$ gives
\[
f_-(q)\leq \frac{\log\hat{q}}{m\log q}f_-^*(\hat{q}).
\]
Hence, the lemma follows.
\end{proof}

For the right local dimensional function $f_+$ we have a similar relationship, but with a subtle difference for the domain of $q$.

\begin{lemma} \label{lem:f-plus-bridge}
Let $J=[q_L,q_R]$ be a relative plateau generated by a word $a_1\dots a_m$, and $q\in\overline{\ub}\cap(q_L,q_R)$. Then
\[
f_+(q)=\frac{\log \hat{q}}{m\log q}f_+^*(\hat{q}),
\]
where $\hat{q}:=\hat{\Phi}_J(q)$.
\end{lemma}

\begin{proof}
The proof is analogous to that of Lemma \ref{lem:f-minus-bridge}. If $q\in\ub$, then we can approximate $q$ from the right by a sequence of points $(r_i)$ from $\vb\cap J$, and use the H\"older properties of $\hat{\Phi}_J$ and $\hat{\Phi}_J^{-1}$ in much the same way as before. On the other hand, if $q\in\overline{\ub}\backslash\ub$, then $q$ is a left endpoint of some relative plateau inside $J$. In this case, $\hat{q}$ is the left endpoint of an entropy plateau in $(1,2]$, and we have $f_+(q)=f_+^*(\hat{q})=0$, so the identity in the lemma holds trivially.
\end{proof}

 Motivated by \cite{Allaart-Baker-Kong-17} we introduce the \emph{left and right bifurcation sets} $\bb_L$ and $\bb_R$, defined by
\begin{equation}\label{eq:left-right-bifurcation set}
\begin{split}
&\bb_L:=\set{q\in(1,M+1]: h(\us_p)\neq h(\us_q)\quad\textrm{for all } p<q},\\
&\bb_R:=\set{q\in(1, M+1]: h(\us_r)\neq h(\us_q)\quad \textrm{for all }r>q}.
\end{split}
\end{equation}
Then $\bb\subset\bb_L$ and $\bb\subset\bb_R$. Furthermore, any $q\in\bb_L\setminus\bb$ is a left endpoint of an entropy plateau, and any $q\in\bb_R\setminus\bb$ is a right endpoint of an entropy plateau. As usual, when $M=1$ we write $\bb^*_L=\bb_L$ and $\bb^*_R=\bb_R$. {Below, we will need the following extension of Proposition \ref{prop:local dimension-B}, which follows from the main results of \cite{Allaart-Baker-Kong-17}.

\begin{proposition}[\cite{Allaart-Baker-Kong-17}] \label{prop:left-and-right-bifurcation-results} \mbox{}

\begin{enumerate}[(i)]
\item If $q\in\bb_L$, then $f_-(q)=\dim_H \u_q$.
\item If $q\in\bb_R$, then $f_+(q)=\dim_H \u_q$.
\end{enumerate}
\end{proposition}
}

\begin{proof}[Proof of Theorem \ref{main1}]
Note by (\ref{eq:dimension-local-u}) that for any $q\in\cb_\f$ we have $f(q)=f_-(q)=f_+(q)=0$. Suppose $q\in\cb_0$, i.e., $q$ is a de Vries-Komornik number. Then $f_-(q)=0$ since $\ub\cap(q-\ep,q)=\emptyset$ for sufficiently small $\ep>0$. Furthermore, $q=q_c(J)$ for some relative plateau $J$, so $\hat{\Phi}_J(q)=q_{KL}^*$. Since $f_+^*(q_{KL}^*)=0$ (see \cite[Theorem 2]{Allaart-Baker-Kong-17}) and $q\in\overline{\ub}$, it follows by Lemma \ref{lem:f-plus-bridge} that $f_+(q)=0$. 
Thus, the proof will be complete once we establish (ii).

Consider first $f_-$. Take $q\in\overline{\ub}\backslash\cb$, and let $J=[q_L,q_R]$ be the smallest relative plateau such that $q\in(q_L,q_R]$.
If $J=[1,M+1]$, then $q\in\bb_L$ and {by Proposition \ref{prop:left-and-right-bifurcation-results}(i),}
\[
f_-(q)=\dim_H \u_q=\frac{h(\us_q)}{\log q}=\frac{h(\wus_q(J))}{\log q}>0.
\]

Otherwise, put $\hat{q}:=\hat{\Phi}_J(q)$. Then $\hat{q}\in\bb_L^*$, so
{by Proposition \ref{prop:left-and-right-bifurcation-results}(i)} it follows that 
\[
f_-^*(\hat{q})=\dim_H \u_{\hat q}^*=\frac{h(\us_{\hat q}^*)}{\log \hat{q}}=\frac{h(\was_{\hat q})}{\log\hat{q}}>0. 
\]
Hence, Lemma \ref{lem:f-minus-bridge} along with Proposition \ref{th:characterization of hat-Phi-J}(iii) gives
\[
f_-(q)=\frac{h(\was_{\hat q})}{m\log q}=\frac{h(\wus_q(J))}{\log q}>0.
\]

Consider next $f_+$. Take again $q\in\overline{\ub}\backslash\cb$. If $q\in\overline{\ub}\backslash\ub$, then $q$ is a left endpoint of a relative plateau and $f_+(q)=0$. So assume $q\in\ub\backslash\cb$. Let $J=[q_L,q_R]$ now be the smallest relative plateau such that $q\in(q_L,q_R)$. 
If $J=[1,M+1]$, then $q\in\bb_R$ and {by Proposition \ref{prop:left-and-right-bifurcation-results}(ii),}
\[
f_+(q)=\dim_H \u_q=\frac{h(\us_q)}{\log q}=\frac{h(\wus_q(J))}{\log q}>0.
\] 
Otherwise, put $\hat{q}:=\hat{\Phi}_J(q)$. Then $\hat{q}\in\bb_R^*$, and
{using Proposition \ref{prop:left-and-right-bifurcation-results}(ii) and Lemma \ref{lem:f-plus-bridge} it follows in the same way as above that 
\[
f_+(q)=\frac{h(\wus_q(J))}{\log q}>0.
\]
}

The statement about $f(q)$ is a direct consequence of the statements about $f_-$ and $f_+$.
\end{proof}

We next prepare to prove Theorem \ref{main2}. Fix a relative plateau $J=[q_L, q_R]$ generated by $\sv=a_1\ldots a_m$.  Recall from Section \ref{s1} that the bases $q_G(J), q_F(J)\in J$ satisfy 
\[
\al(q_G(J))=\left(\sv^+\overline{\sv^+}\right)^\f \quad \textrm{and} \quad \al(q_F(J))=\left(\sv^+\overline{\sv}\overline{\sv^+}\sv\right)^\f.
\]
Furthermore, the de Vries-Komornik number $q_c(J)=\min(\ub\cap J)$ satisfies
\[
\al(q_c(J))=\sv^+\overline{\sv}\overline{\sv^+}\sv^+\overline{\sv^+}\sv\sv^+\overline{\sv}\cdots.
\]
By Lemma \ref{lem:characterization of V-U}, the bases $q_G(J), q_F(J)$ and $q_c(J)$ all belong to $\vb\cap(q_L, q_R]$, so we may define their image bases in $\vb^*$ by
\[
\hat q_G:=\hat\Phi_J(q_G(J)), \quad\hat q_F:=\hat\Phi_J(q_F(J))\quad\textrm{and}\quad \hat q_c:=\hat\Phi_J(q_c(J)).
\]
The quasi-greedy expansions of these bases are given by 
\[
\al^*(\hat q_G)=(10)^\f,\quad \al^*(\hat q_F)=(1100)^\f,\quad\textrm{and}\quad \al^*(\hat q_c)=11010011\;00101101\cdots.
\]
We have $\hat q_G={(1+\sqrt{5})/2}\approx 1.61803, \hat q_F\approx 1.75488$ and $\hat q_c\approx 1.78723$. Note that $\hat{q}_c$ is simply the Komornik-Loreti constant $q_{KL}^*$. The following result is due to Glendinning and Sidorov \cite{Glendinning_Sidorov_2001} and Komornik et al.~\cite{Komornik-Kong-Li-17}; see also \cite{Allaart-Kong-2018}.

\begin{proposition}
\label{prop:unique expansion-two digits case}
Let $q\in(1,2]$. Then the entropy function 
\[H: q\mapsto h(\wus_q^*)\] is a Devil's staircase, i.e., $H$ is continuous, non-deceasing and locally constant almost everywhere on $(1,2]$.

\begin{enumerate}[{\rm(i)}]
\item If $1<q\le \hat q_G$, then $\wus_q^*=\emptyset$.

\item If $\hat q_G<q\le \hat q_F$, then $\wus_q^*=\big\{(01)^\f, (10)^\f\big\}$.

\item If $\hat q_F<q< \hat q_{c}$, then $\wus_q^*$ is countably infinite.

\item If $q=\hat q_{c}$, then $\wus_q^*$ is uncountable but $h(\wus_q^*)=0$.

\item If $\hat q_c<q\le 2$, then $h(\wus_q^*)>0$.
\end{enumerate}
\end{proposition}

\begin{proof}[Proof of Theorem \ref{main2}]
Recall from Proposition \ref{th:characterization of hat-Phi-J} (iii) that for each $q\in\vb\cap(q_L, q_R]$ we have
\begin{equation} \label{eq:jh-4}
\wus_q(J)=\Phi_J^{-1}\left(\big\{(x_i)\in\was_{\hat q}: x_1=1\big\}\right)\quad\textrm{and}\quad h(\wus_q(J))=\frac{h(\was_{\hat q})}{m},
\end{equation}
where $\hat{q}:=\hat{\Phi}_J(q)$. Since $\bb^*\subset\ub^*$, the function $q\mapsto h(\was_q)$ is constant on each connected component of $(1,2]\backslash \ub^*$. Recalling from Proposition \ref{th:characterization of hat-Phi-J} (ii) that $\hat{\Phi}_J(\ub\cap J)=\ub^*$, it follows by \eqref{eq:jh-4} that the function $H_J: q\mapsto h(\wus_q(J))$ is constant on each connected component of $(q_L,q_R]\backslash (\ub\cap J)$. Since $\ub$ is Lebesgue null, this implies that $H_J$ is almost everywhere locally constant on $J$. That $H_J$ is also continuous follows since $\ub\cap J$ has no isolated points, and the  restriction of $H_J$ to $\ub\cap J$ is the composition of the map $q\mapsto h(\was_q)$ with $\hat{\Phi}_J$; the former is continuous by Proposition \ref{prop:unique expansion-two digits case}, the latter by Lemma \ref{lem:Holder continuity of hatPhi}. Therefore, the entropy function $H_J$ is a Devil's staircase.

Statements (i)-(v) of Theorem \ref{main2} now follow from the corresponding statements of Proposition \ref{prop:unique expansion-two digits case}. For example, if $q_L<q\le q_G(J)$, then by (\ref{eq:jh-4}) it follows that
\[
\wus_q(J)\subset\wus_{q_G(J)}(J)=\Phi_J^{-1}\left(\big\{(x_i)\in\was_{\hat q_G}: x_1=1\big\}\right)=\emptyset,
\]
where the last equality follows from Proposition \ref{prop:unique expansion-two digits case} (i). 

Similarly, for (ii) we take $q\in(q_G(J), q_F(J)]$. Then by (\ref{eq:jh-4}) and Proposition \ref{prop:unique expansion-two digits case} (ii) it follows that 
\[
\wus_q(J)\subset\wus_{q_F(J)}(J)=\Phi_J^{-1}\left(\big\{(x_i)\in\was_{\hat q_F}: x_1=1\big\}\right)=\Phi_J^{-1}(\big\{(10)^\f\big\})=\big\{(\sv^+\overline{\sv^+})^\f\big\}.
\]
Vice versa, one checks easily using \eqref{eq:def-widetilde-uq} that $\big\{(\sv^+\overline{\sv^+})^\f\big\}\subset\wus_q(J)$.

For (iii), we take $q\in (q_F(J),q_c(J))$. Then 
\[
\left\{\big(\sv^+\overline{\sv^+}\big)^k\big(\sv^+\overline{\sv}\overline{\sv^+}\sv\big)^\f: k\in\N\right\}\subset \wus_q(J).
\]
On the other hand, we can find a sequence $(q_n)$ in $\vb$ that converges from the left to $q_c(J)$: if $\alpha(q_c(J))=\theta_1\theta_2\dots$, we can take $q_n$ with $\alpha(q_n)=(\theta_1\dots \theta_{2^nm}^-)^\f$. Then for large enough $n$, $q<q_n$ and $\hat{q}_n:=\hat{\Phi}_J(q_n)<q_{KL}^*$, so $\wus_q(J)$ is countable by Proposition \ref{prop:unique expansion-two digits case} (iii) and \eqref{eq:jh-4}.

Statement (iv) is immediate from \eqref{eq:jh-4} and Proposition \ref{prop:unique expansion-two digits case}(iv), since $q_c(J)\in\vb$.

%The proofs of (iii) and (iv) are similar to those of (i) and (ii).

For (v) we first note that $q_c(J)\in\vb\cap(q_L, q_R]$ and $\hat q_c\in\vb^*$. Furthermore, there exists a sequence $(r_i)$ in $\vb\cap(q_L, q_R]$ such that $r_i\searrow q_c(J)$. This follows from Lemma \ref{lem:dense-plateaus}, since the endpoints of relative plateaus lie in $\vb$.
Accordingly, the image sequence $(\hat r_i)$ in $\vb^*$ satisfies $\hat r_i\searrow \hat q_c$, where $\hat r_i=\hat \Phi_J(r_i)$. So, for any $q\in(q_c(J), q_R]$ there exists $r_i\in\vb\cap (q_c(J), q)$ such that 
\[
\wus_q(J)\supset\wus_{r_i}(J)\quad\textrm{and}\quad h(\wus_{r_i}(J))=\frac{h(\was_{\hat r_i})}{m}>0.
\]
This proves (v). 
\end{proof}

\begin{proof}[Proof of Corollary \ref{cor:variation-in-plateau}]
Take $q\in(q_L, q_R]$. Then $\al_1(q)\ldots \al_m(q)=a_1\ldots a_m^+$. Note that $\al(q_L)=(a_1\ldots a_m)^\f$. Then by the definitions of $\wus_q$  and $\wus_{q_L}$ it follows that $\wus_q(J)\subset \wus_q\setminus\wus_{q_L}$. Furthermore,  any sequence $(x_i)\in\wus_q\setminus\wus_{q_L}$ or its reflection $\overline{(x_i)}$ has a tail sequence in $\wus_q(J)\cup\big\{\sv^\f\big\}$. Therefore,  
\[
\dim_H(\wus_q\setminus\wus_{q_L})= \dim_H\wus_q(J).
\]
Hence, the result follows from Theorem \ref{main2}.
\end{proof}

\begin{proof}[Proof of Corollary \ref{cor:continuity-of-f}]
Fix $q_0\in\overline{\ub}$. If $q_0\in\cb_\f$, the same argument based on \eqref{eq:dimension-local-u} that we used to prove $f(q_0)=0$ shows also that $f(q)\to 0$ as $q\to q_0$. Hence $f$ is continuous at $q_0$. If $q_0\in\cb_0$, then $q_0=q_c(J)$ for some relative plateau $J$. Since $\wus_q(I)\subset \wus_q(J)$ whenever $I\subset J$, Theorem \ref{main1} implies that
\[
f(q)\leq \frac{h(\wus_q(J))}{\log q}, \qquad\mbox{for all $q\in J, q\neq q_0$}.
\]
But by Theorem \ref{main2}, $h(\wus_q(J))\to 0$ as $q\to q_0=q_c(J)$. Hence, $f(q)\to 0=f(q_0)$. This shows that $f$ is continuous on $\cb$.

Now suppose $q_0\in\overline{\ub}\backslash\cb$. Then, using Lemma \ref{lem:dense-plateaus}, there is a sequence of relative plateaus $[p_L(i),p_R(i)]$ such that $p_L(i)\nearrow q_0$ as $i\to\infty$. Each of these plateaus contains a point $q_i\in\cb$ (in fact, infinitely many), so that $q_i\nearrow q_0$. By Theorem \ref{main1}, we obtain $f(q_0)>0=\lim_{i\to \infty} f(q_i)$. Therefore, $f$ is discontinuous at $q_0$. The corresponding statements for $f_-$ and $f_+$ follow in the same way.
\end{proof}

Finally, we prove Theorem \ref{main-b}.

\begin{proof}[Proof of Theorem \ref{main-b}]
(i) Let $J=J_{\mathbf i}=[q_L, q_R]$ be a relative plateau generated by $\sv=a_1\ldots a_m$. We show that the next level relative plateaus $J_{\mathbf{i}j}$, $j=1,2,\dots$ are exactly the maximal intervals on which $h(\wus_q(J))$ is positive and constant; this, along with Theorem \ref{main2}, will imply (i). Fix $j\in\N$, and write $I:=J_{\mathbf{i}j}=[p_L,p_R]$. Then $p_L,p_R\in\vb$, so we may put $\hat{p}_L:=\hat{\Phi}_J(p_L)$ and $\hat{p}_R:=\hat{\Phi}_J(p_R)$. Then $\hat{I}:=[\hat{p}_L,\hat{p}_R]$ is an entropy plateau in $(1,2]$, and so $h(\wus_{\hat q}^*)$ is positive and constant on $\hat{I}$. By Proposition \ref{th:characterization of hat-Phi-J}(iii), it follows that $h(\wus_q(J))$ is positive and constant on $I$. 

By Lemma \ref{lem:dense-plateaus} the union $\bigcup_{j\in\N}J_{\mathbf{i}j}$ is dense in $(q_c(J),q_R]$. As a result, $I$ is a {\em maximal} interval on which $h(\wus_q(J))$ is constant.

(ii) Since $\bigcup_{j\in\N}J_{\mathbf{i}j}$ is dense in $(q_c(J),q_R]$, each $q\in\bb(J)$ is an accumulation point of the set of endpoints of the intervals $J_{\mathbf{i}j}$. Since these endpoints lie in $\vb$ and $\vb$ is closed, it follows that $\bb(J)\subset \vb$. Hence, $\hat{\Phi}_J(q)$ is well defined for all $q\in\bb(J)$. It now follows immediately from part (i) that
\begin{equation}
\hat{\Phi}_J(\bb(J))=\bb^*.
\label{eq:bifurcation-bridge}
\end{equation}
Since $\bb^*\subset \ub^*$, it follows from Proposition \ref{th:characterization of hat-Phi-J}(ii) that $\bb(J)\subset \ub\cap J$.

(iii) That $\bb(J)$ is Lebesgue null is now obvious from (ii), since $\ub$ is Lebesgue null.

(iv) Note by (\ref{eq:dimension-local-u}) that 
\[
\dim_H(\ub\cap J)=\frac{\log 2}{m\log q_R}. 
\]
Since $\bb(J)\subset\ub\cap J$, it therefore suffices to prove 
\begin{equation}\label{eq:k15-1}
\dim_H\bb(J)\ge \frac{\log 2}{m\log q_R}.
\end{equation}
Observe that $2=\hat\Phi_J(q_R)\in\bb^*$. Furthermore, the proof of \cite[Theorem 3]{AlcarazBarrera-Baker-Kong-2016} shows that $\dim_H(\bb^*\cap[2-\eta,2])=1$ for every $\eta>0$. Given $\ep>0$, we can choose a point $q_0\in \vb\cap(q_R-\ep,q_R)$; let $\hat{q}_0:=\hat{\Phi}_J(q_0)$.
By Lemma \ref{lem:Holder continuity of hatPhi}, $\hat{\Phi}_J$ is H\"older continuous with exponent $\log \hat{q}_0/(m\log q_R)$ on $[q_0,q_R]$, and so, using \eqref{eq:bifurcation-bridge},
\[
1=\dim_H(\bb^*\cap[\hat{q}_0,2])=\dim_H \hat{\Phi}_J(\bb(J)\cap[q_0,q_R])\leq \frac{m\log q_R}{\log \hat{q}_0}\dim_H \bb(J).
%\dim_H\bb(J)=\dim_H(\bb(J)\cap (q_L, q_R))=\frac{\log 2}{m\log q_R}\dim_H(\bb^*\cap(1, 2))=\frac{\log 2}{m\log q_R}.
\]
Letting $\ep\to 0$, $\hat{q}_0\to 2$ and we obtain (\ref{eq:k15-1}), as desired. 

(v) By (i) and the countable stability of Hausdorff dimension,
\[
\dim_H\big((\ub\cap J)\backslash \bb(J)\big)=\sup_{j\in\N} \dim_H(\ub\cap J_{\mathbf{i}j}).
\]
If $J_{\mathbf{i}j}=[p_L,p_R]$ is generated by the block $b_1\dots b_l$, then
\begin{equation}
\dim_H(\ub\cap J_{\mathbf{i}j})=\frac{\log 2}{l\log p_R}.
\label{eq:relative-plateau-dimension}
\end{equation}
Furthermore, $b_1\dots b_l$ must be a concatenation of words from $\mathcal{L}=\big\{\sv,\sv^+,\overline{\sv},\overline{\sv^+}\big\}$ , so $l$ is a multiple of $m$. Since $b_1\dots b_l$ is admissible and $\alpha(p_L)>q_c(J)$, it follows from \eqref{eq:q_c} that $l\geq 3m$. (See Figure \ref{fig1}.) Moreover, the only relative plateau among the $J_{\mathbf{i}j}$ with $l=3m$ is the one with generating word $b_1\dots b_l=\sv^+\overline{\sv}\overline{\sv^+}$, whose right endpoint is $p_0$. 

It remains to check that this plateau maximizes the expression in \eqref{eq:relative-plateau-dimension}. To this end, take any other relative plateau $[p_L,p_R]\subset J$ generated by a block of length $l=km$. If $p_R\geq p_0$, then $l\log p_R\geq 3m\log p_0$. On the other hand, suppose $p_R<p_0$. Then $\alpha(q_c(J))\prec \alpha(p_L)\prec \big(\sv^+\overline{\sv}\overline{\sv^+}\big)^\f$, and since $\alpha(p_L)$ must correspond to an infinite path in the labeled digraph $\mathcal{G}=(G,\mathcal{L})$ from Figure \ref{fig1}, this is only possible when $k\geq 5$. In \cite{Allaart-Baker-Kong-17} it was observed that $q_{KL}\geq (M+2)/2$. Estimating $p_R$ below by $q_{KL}$ and $p_0$ above by $M+1$, we thus obtain for all $M\geq 2$,
\begin{equation*}
l\log p_R \geq 5m\log q_{KL} \geq 5m\log\left(\frac{M+2}{2}\right) \geq 3m\log(M+1)>3m\log p_0,
\end{equation*}
where we used the algebraic inequality $(M+2)^5\geq 32(M+1)^3$, valid for $M\geq 2$. For the case $M=1$ we can use the better estimate $q_{KL}>1.78$, giving $5\log q_{KL}>2.8>3\log 2>3\log p_0$, where we have used the natural logarithm. Thus, in all cases, $l\log p_R\geq 3m\log p_0$, as was to be shown.
\end{proof}

\begin{remark}
Note by \eqref{eq:bifurcation-bridge} and Proposition \ref{th:characterization of hat-Phi-J}(i) that the {relative bifurcation sets $\bb(J_{\mathbf i}): \mathbf{i}\in\{1,2,\dots\}^n$}, $n\in\N$ are mutually homeomorphic.
\end{remark}

To end this section, we illustrate how Theorem \ref{main1} can be combined with the entropy ``bridge" of Proposition \ref{th:characterization of hat-Phi-J}(iii) to compute $f(q)$ explicitly at some special points.

\begin{example} \label{ex:endpoints}
Let $J=[p_L,p_R]$ be a relative plateau generated by the word $\sv=a_1\dots a_m$. For any integer $k\geq 3$, let $[q_L,q_R]$ be the relative plateau generated by the admissible word $\mathbf{b}:=\sv^+\overline{\sv}^{k-2}\overline{\sv^+}$. Then $[q_L,q_R]\subset J$, and $J$ is the parent interval of $[q_L,q_R]$. 
Note that $q_L\in\vb$. Hence, by Theorem \ref{main1} and Proposition \ref{th:characterization of hat-Phi-J}(iii),
\[
f(q_L)=f_-(q_L)=\frac{h\big(\wus_{q_L}(J)\big)}{\log q_L}=\frac{h\big(\wus_{\hat{q}_L}^*\big)}{m\log q_L},
\]
where $\hat{q}_L:=\hat{\Phi}_J(q_L)$. Note that 
\[
\alpha^*(\hat{q}_L)=\Phi_J(\alpha(q_L))=\Phi_J\left(\big(\sv^+\overline{\sv}^{k-2}\overline{\sv^+}\big)^\f\right)=(1^{k-1}0)^\f.  
\]
Define the sets
\[
\wvs_{\hat{q}}^*:=\big\{(x_i)\in\{0,1\}^\N: \overline{\al^*(\hat{q})}\preceq\si^n((x_i))\preceq \al^*(\hat{q})\ \forall n\ge 0\big\}, \qquad \hat{q}\in(1,2].
\]
It is well known (see \cite{Komornik-Kong-Li-17} or \cite{Allaart-Kong-2018}) that $h\big(\wus_{\hat{q}}^*\big)=h\big(\wvs_{\hat{q}}^*\big)$.
Moreover, $\wvs_{\hat{q}_L}^*$ is a subshift of finite type and it consists of precisely those sequences in $\{0,1\}^\N$ which do not contain the word $1^{k}$ or $0^{k}$. A standard argument (see \cite{Lind_Marcus_1995} or \cite[Lemma 4.2]{Allaart-Baker-Kong-17}) now shows that $h\big(\wvs_{\hat{q}_L}^*\big)=\log\varphi_{k-1}$, where for each $j\in\N$, $\varphi_j$ is the unique root in $(1,2)$ of $1+x+\dots+x^{j-1}=x^j$.
Therefore,
\[
f(q_L)=f_-(q_L)=\frac{\log\varphi_{k-1}}{m\log q_L}.
\]
Of course, $f_+(q_L)=0$. Similarly, since $h\big(\wus_q(J)\big)$ is constant on $[q_L,q_R]$, Theorem \ref{main1} gives
\[
f(q_R)=f_+(q_R)=\frac{\log\varphi_{k-1}}{m\log q_R}.
\]
On the other hand, by \eqref{eq:dimension-local-u},
\[
f_-(q_R)=\frac{\log 2}{mk\log q_R},
\]
since the generating word $\mathbf{b}$ of $[q_L,q_R]$ has length $mk$. Observe that $f_-(q_R)<f_+(q_R)$. This last inequality holds generally, for any relative plateau $[q_L,q_R]$ in $J$: If $[q_L,q_R]$ has generating block $\mathbf{b}$ of length $l$, then $l=mk$ for some $k\in\N$. Again putting $\hat{q}_L:=\hat{\Phi}_J(q_L)$, Lemma 3.1(ii) in \cite{Allaart-Baker-Kong-17} gives
\[
h\big(\wvs_{\hat{q}_L}^*\big)>\frac{\log 2}{k},
\]
and so
\[
f_-(q_R)=\frac{\log 2}{mk\log q_R}<\frac{h\big(\wvs_{\hat{q}_L}^*\big)}{m\log q_R}=\frac{h\big(\wvs_{\hat{q}_R}^*\big)}{m\log q_R}=f_+(q_R).
\]
(There is one exception: If $[q_L,q_R]$ is a first-level relative plateau (i.e. an entropy plateau) generated by $\sv=a_1\dots a_m$, then the parent interval $J$ is $J_\emptyset=[1,M+1]$. In this case, there is no map $\Phi_J$ relating $h(\wus_{q_L}(J))$ to the alphabet $\{0,1\}$. Instead,
\[
f_+(q_R)=\frac{h(\wus_{q_R})}{\log q_R}=\frac{h(\wus_{q_L})}{\log q_R}, \qquad\mbox{and} \qquad 
f_-(q_R)=\frac{\log 2}{m\log q_R}.
\]
As shown in \cite[Lemma 3.1(ii)]{Allaart-Baker-Kong-17}, these two quantities are equal if (and only if) $M=2j+1\geq 3$, and $\sv=a_1:=j+1$.)

The above procedure generalizes to other relative plateaus: $\wvs_{\hat{q}_L}^*$ is always a subshift of finite type of $\{0,1\}^\N$, so its topological entropy can be calculated, numerically at least, by writing down the corresponding adjacency matrix and computing its spectral radius; {see \cite[Chap.~5]{Lind_Marcus_1995}.}
\end{example}

\section{Proof of Theorem \ref{main3}} \label{sec:proof-of-theorem3}

\begin{proof}[Proof of Theorem \ref{main3}]
{Let $1<t_1<t_2\leq M+1$, and let $J=J_{\mathbf i}=[q_L, q_R]$ be the smallest relative plateau containing $[t_1,t_2]$. Define}
\[
g_J(t_1,t_2):=\max\set{\frac{h(\wus_q(J))}{\log q}: q\in\overline{\bb(J)\cap[t_1, t_2]}},
\]
so we need to show that 
\begin{equation}
\dim_H(\ub\cap[t_1,t_2])=g_J(t_1,t_2). 
\label{eq:Th3-equation}
\end{equation}
Note first that, if $t_1=q_L$, then there exists $\delta>0$ such that $\ub\cap[t_1,t_1+\delta]=\emptyset$, and hence $\bb(J)\cap [t_1, t_2]=\bb(J)\cap[t_1+\delta, t_2]$. Therefore, both sides of \eqref{eq:Th3-equation} remain unchanged upon replacing $t_1$ with $t_1+\delta$. Consequently, we may assume that $t_1>q_L$.

We first demonstrate the lower bound. Since $\bb(J)\subset \ub$, we may assume without loss of generality that $\ub\cap(t_1, t_2)\neq\emptyset$. Then by the definition of $J$ we also have $\bb(J)\cap[t_1, t_2]\ne \emptyset$. Since $t_1>q_L$, Theorem \ref{main1} gives for any $q\in\bb(J)\cap[t_1, t_2]$ that
\[
f_-(q)=\frac{h(\wus_q(J))}{\log q}>0.
\]
Since $\bb(J)\cap[t_1, t_2]\subset\ub\cap[t_1, t_2]$, this implies
\begin{equation*}
\dim_H(\ub\cap[t_1, t_2])\ge \sup\set{f_-(q): q\in\bb(J)\cap[t_1, t_2]}=g_J(t_1,t_2),
%&=\max\set{\frac{h(\wus_q(J))}{\log q}: q\in\overline{\bb(J)\cap[t_1, t_2]}},
\end{equation*}
where in the last step we used the continuity of the map $q\mapsto h(\wus_q(J))$ (cf. Theorem \ref{main2}). 

This proves the lower bound. For the upper bound, we use a compactness argument similar to that used in \cite{Kalle-Kong-Li-Lv-2016}.
Recall from Theorem \ref{main-b}(i) that 
\[
(q_L,q_R]=\bb(J)\cup(q_L, q_c(J)]\cup\bigcup_{j=1}^\infty J_{\mathbf{i}j}. 
\]
Let $J_{\mathbf{i}j}=[p_L,p_R]$ be a relative plateau that intersects $[t_1,t_2]$ in more than one point. Then either $p_L$ or $p_R$ lies in $(t_1,t_2)$, so at least one of these two points lies in $\overline{\bb(J)\cap[t_1, t_2]}$. Then by the proof of \cite[Theorem 4.1]{Allaart-Baker-Kong-17} it follows that 
\[
\dim_H(\ub\cap[p_L,p_R])=\frac{\log 2}{m\log p_R}\le \min\set{\frac{h(\wus_{p_L})}{\log p_L}, \frac{h(\wus_{p_R})}{\log p_R}} \le g_J(t_1,t_2).
\]
By the countable stability of Hausdorff dimension, we obtain
\begin{equation}
\dim_H\left(\ub\cap \bigcup_{j=1}^\infty J_{\mathbf{i}j}{\cap[t_1,t_2]}\right)\leq g_J(t_1,t_2).
\label{eq:bound-1}
\end{equation}
Now let $\ep>0$. Then for each $q\in \overline{\bb(J)\cap[t_1, t_2]}$ there is a number $\delta(q)>0$ such that
\[
\dim_H\big(\ub\cap(q-\delta(q),q+\delta(q))\big)\leq f(q)+\ep \leq \frac{h(\wus_q(J))}{\log q}+\ep\leq g_J(t_1,t_2)+\ep.
\]
The intervals $(q-\delta(q),q+\delta(q))$ form an open cover of $\overline{\bb(J)\cap[t_1, t_2]}$, and since $\overline{\bb(J)\cap[t_1, t_2]}$ is compact, this open cover contains a finite subcover. Therefore,
\begin{equation}
\dim_H\left(\ub\cap \overline{\bb(J)\cap[t_1, t_2]}\right)\leq g_J(t_1,t_2)+\ep.
\label{eq:bound-2}
\end{equation}
Letting $\ep\to 0$, \eqref{eq:bound-1} and \eqref{eq:bound-2} together give the upper bound in \eqref{eq:Th3-equation}, since $\ub\cap(q_L,q_c(J))=\emptyset$.
\end{proof}

\section{Proof of Theorem \ref{main4}} \label{sec:proof-of-theorem4}  

Recall the definitions \eqref{eq:strongly-univoque-set} and \eqref{eq:Wq-def} of $\check{\us}_q$ and $\ws_q$, and that $\mathcal{W}_q=\pi_q(\ws_q)$. When $M=1$ we write $\ws_q^*:=\ws_q$. We will prove Theorem \ref{main4} indirectly, by showing that $\dim_H\mathcal W_q=0$ for $q\in\cb$, and if $q\in\overline{\ub}\backslash \cb$, then
\[
\dim_H\mathcal W_q=\frac{h(\wus_q(J))}{\log q},
\]
where $J=[q_L, q_R]$ is the smallest relative plateau such that $q\in(q_L, q_R]$. The result then follows from Theorem \ref{main1}.

Recall that on the sequence space $\Omega_M$ we are using the metric $\rho$ from \eqref{eq:rho-metric}. The following lemma allows us to work with subsets of $\Omega_M$ rather than sets in Euclidean space.

\begin{lemma} \label{lem:bi-Lipschitz}
Let $q\in(1,M+1]$. For any subset $F\subset \wus_q$, we have
\[
\dim_H \pi_q(F)=\frac{\log 2}{\log q}\dim_H F.
\]
\end{lemma}

\begin{proof}
It is well known (see \cite[Lemma 2.7]{Jordan-Shmerkin-Solomyak-2011} or \cite[Lemma 2.2]{Allaart-2017}) that $\pi_q$ is bi-Lipschitz on $\wus_q$ with respect to the metric
\[
\rho_q((x_i), (y_i)):=q^{-\inf\set{i\ge 0: x_{i+1}\ne y_{i+1}}}.
\]
Hence, with respect to the metric $\rho_q$ on $\Omega_M$, $F$ and $\pi_q(F)$ have the same Hausdorff dimension for any $F\subset \wus_q$. The lemma now follows since $\rho_q=\rho^{\log q/\log 2}$.
\end{proof}

In view of Lemma \ref{lem:bi-Lipschitz}, it suffices to compute $\dim_H \ws_q$. The next lemma facilitates this.

\begin{lemma} \label{lem:W-bridge}
Let $J=[q_L,q_R]$ be a relative plateau generated by $\sv=a_1\dots a_m$, and $q\in\vb\cap(q_L,q_R]$. Then
\[
\dim_H \ws_q=\frac{1}{m}\dim_H \ws_{\hat{q}}^*,
\]
where $\hat{q}:=\hat{\Phi}_J(q)$.
\end{lemma}

\begin{proof}
Since $\ws_q\subset\wus_q$ and every sequence in $\ws_q$ must eventually contain the word $\alpha_1(q)\dots\alpha_m(q)$, we have
\begin{equation*}
\dim_H \ws_q=\dim_H (\ws_q\cap \wus_q(J)).
%\label{eq:same-dimension-really}
\end{equation*}
By a trivial extension of Proposition \ref{th:characterization of hat-Phi-J}(iii),
\[
\Phi_J\big(\ws_q\cap \wus_q(J)\big)=\ws_{\hat{q}}^*\cap \wus_{\hat{q}}^*(1),
\]
where $\wus_{\hat{q}}^*(1):=\{(x_i)\in \wus_{\hat{q}}^*: x_1=1\}$.
Since $\Phi_J$ is bi-H\"older continuous with exponent $1/m$, it follows that
\[
\dim_H \ws_q=\frac{1}{m}\dim_H\left(\ws_{\hat{q}}^*\cap \wus_{\hat{q}}^*(1)\right)=\frac{1}{m}\dim_H \ws_{\hat{q}}^*,
\]
as desired.
\end{proof}

We first consider the case when $q\in\cb$.

\begin{proposition} \label{prop:W-q-C}
If $q\in\cb$, then $\dim_H \ws_q=0$.
\end{proposition}

\begin{proof}
If $q=q_{KL}$, then $\dim_H\ws_q\leq\dim_H \wus_q=0$ by Proposition \ref{prop:unique expansion-two digits case}, which holds also for larger alphabets (cf.~\cite{Kong_Li_Dekking_2010}). And if $q\in\cb_0\backslash \{q_{KL}\}$, then $q=q_c(J)$ for some relative plateau $J$, so that $\hat{q}:=\hat{\Phi}_J(q)=q_{KL}^*$ and the result follows from Lemma \ref{lem:W-bridge} and Proposition \ref{prop:unique expansion-two digits case}.

Suppose $q\in\cb_\f$. Then $q\in\ub\subset \vb$ by Proposition \ref{prop:property of C-infity}, and there are infinitely many relative plateaus $J=[q_L,q_R]$ such that $q\in(q_L,q_R]$. If $J$ is one such relative plateau generated by a word of length $m$, then Lemma \ref{lem:W-bridge} gives
\[
\dim_H \ws_q=\frac{1}{m}\dim_H \ws_{\hat{q}}^*\leq \frac{1}{m}\dim_H \{0,1\}^\N=\frac{1}{m}.
\]
Letting $m\to\infty$, we obtain $\dim_H \ws_q=0$.
\end{proof}

Recall the definition of $\bb_L$ (and $\bb_L^*$) from \eqref{eq:left-right-bifurcation set}.

\begin{proposition} \label{prop:W-q-B}
Let $q\in\bb_L$. Then
\[
\dim_H\ws_q=\frac{h(\wus_q)}{\log 2}=\dim_H \wus_q. 
\]
\end{proposition}

The proof uses the following lemma.

\begin{lemma} \label{lem:number-of-paths}
Let $G=(V,E)$ be a strongly connected directed graph with adjacency matrix $A$, and let $\gamma$ be the spectral radius of $A$. 
Let $\mathbf P_k^{u, v}$ be the set of all directed paths of length $k$ in $G$ starting from vertex $u$ and ending at vertex $v$. Then there are constants $0<C_1<C_2$ such that the following hold:
\begin{enumerate}[(i)]
\item For each vertex $v\in V$ and for each $K\in\N$, there is an integer $k\geq K$ such that 
\[
\#\mathbf P_k^{v, v}\geq C_1\gamma^k.
\]
\item For all $k\in\N$,
\[
\sum_{u,v\in V} \#\mathbf P_k^{u, v}\leq C_2\gamma^k.
\]
\end{enumerate}
\end{lemma}

\begin{proof}
By the Perron-Frobenius theorem, $\gamma$ is an eigenvalue of $A$ and there is a strictly positive left eigenvector $\mathbf{v}=[v_1 \dots v_N]$ of $A$ corresponding to $\gamma$, where $N:=\#V$. We may normalize $\mathbf{v}$ so that $\max v_i=1$. Set
\[
C_1:=\frac{\min_i v_i}{N\gamma^N}.
\]
Clearly, for any two vertices $u$ and $v$ in $V$, there is a path from $u$ to $v$ of length at most $N$. Fix $v\in V$ and $K\in\N$. Without loss of generality order $V$ so that $v$ is the first vertex. 
{Let $\mathbf{e}_1=[1\ 0\ \dots\ 0]^T$ be the first standard unit vector in $\R^N$, and let $\mathbf{1}=[1\ 1\ \dots\ 1]$ be the row vector of all $1$'s in $\R^N$.
}
The number of paths in $G$ of length $K$ starting anywhere in $G$ but ending at $v$ is
\[
\mathbf{1}A^K \mathbf{e}_1\geq \mathbf{v}A^K\mathbf{e}_1=\gamma^K \mathbf{v}\mathbf{e}_1\geq \gamma^K\min v_i.
\]
Hence there is a vertex $u$ in $V$ such that %the number of paths in $G$ from $u$ to $v$ of length $K$ is at least
\[
\#\mathbf P_K^{u, v}\geq N^{-1}\gamma^K \min v_i.
\]
Let $L$ be the length of the shortest path in $G$ from $v$ to $u$; then $L\leq N$. Set $k:=K+L$. It follows that %the number of paths from $v$ to itself of length $k$ is at least
\[
\#\mathbf P_k^{v, v}\geq N^{-1}\gamma^K \min v_i=N^{-1}\gamma^{k-L}\min v_i\geq \frac{\min v_i}{N\gamma^N}\gamma^k=C_1\gamma^k.
\]
This proves (i). The proof of (ii) is standard (cf.~\cite[Chap.~4]{Lind_Marcus_1995}).
\end{proof}

Recall from \cite{Komornik_Loreti_2002} that the Komornik-Loreti constant $q_{KL}=q_{KL}(M)$ satisfies
\begin{equation} \label{eq:lambda}
\al(q_{KL})=\la_1\la_2\ldots,
\end{equation}
where for each $i\ge 1$,
\begin{equation*} 
%\label{eq:lambda-i}
\la_i=\la_i(M):=\begin{cases}
k+\tau_i-\tau_{i-1} & \qquad\textrm{if \quad$M=2k$},\\
k+\tau_i & \qquad\textrm{if \quad$M=2k+1$}.
\end{cases}
\end{equation*}
Here $(\tau_i)_{i=0}^\f=0110100110010110\ldots$ is the classical Thue-Morse sequence.  

In the proof below we use the sets
\[
\wvs_q:=\big\{(x_i)\in\Omega_M: \overline{\al(q)}\preceq\si^n((x_i))\preceq \al(q)\ \forall n\ge 0\big\}, \qquad q\in(1,M+1].
\]
It is well known (see \cite{Komornik-Kong-Li-17} or \cite{Allaart-Baker-Kong-17}) that $\dim_H \wus_q=\dim_H \wvs_q$ for every $q$.

\begin{proof}[Proof of Proposition \ref{prop:W-q-B}]
Fix $q\in\bb_L$. Then $q>q_{KL}$, so $\alpha(q)\succ\la_1\la_2\dots$, and hence there is a number $l_0\geq 1$ such that $\alpha_1\dots\alpha_{l_0-1}=\la_1\dots\la_{l_0-1}$ and $\alpha_{l_0}>\la_{l_0}$, where for brevity we put $\alpha_i:=\alpha_i(q)$.

By \cite[Lemma 3.16]{AlcarazBarrera-Baker-Kong-2016} (see also \cite{Allaart-Baker-Kong-17}), there is an increasing sequence $(l_n)$ of integers with $l_n>l_0$ such that for each $n$, there is an entropy plateau $[p_L(n),p_R(n)]$ with $\alpha(p_L(n))=(\alpha_1\dots\alpha_{l_n}^-)^\f$, and moreover $p_L(n)\nearrow q$. By the continuity of the function $p\mapsto \dim_H \wus_p$ it is enough to prove that $\dim_H \ws_q\geq \dim_H \wus_{p_L(n)}$ for each $n$. 

Fix therefore an integer $n$, and put $p:=p_L(n)$, and $l:=l_n$. Then $\wvs_p$ is a subshift of finite type, characterized by 
\[
(x_i)\in \wvs_p \qquad\Leftrightarrow \qquad \overline{\alpha_1\dots \alpha_l}\prec x_{k+1}\dots x_{k+l}\prec \alpha_1\dots \alpha_l \quad\forall k\geq 0.
\]
We represent $\wvs_p$ by a labeled directed graph $G=(V,E,L)$ in the usual way: the set $V$ of vertices consists of allowed words in $\wvs_p$ of length $l-1$, and there is an edge $uv$ from $u=x_1\dots x_{l-1}$ to $v=y_1\dots y_{l-1}$ if and only if $x_2\dots x_{l-1}=y_1\dots y_{l-2}$ and $x_1\dots x_{l-1}y_{l-1}$ is an allowed word in $\wvs_p$, in which case we label the edge $uv$ with $y_{l-1}$.

Assume first that $\wvs_p$ is transitive, so the graph $G$ is strongly connected. {Let $\gamma$ be the spectral radius of the adjacency matrix of $G$, and} let $C_1,C_2$ be the constants from Lemma \ref{lem:number-of-paths}(i). {Put $C:=\max\{C_2,C_1^{-1}\}$.} Let $\mathbf{u}=\alpha_1\dots\alpha_{l-1}$ be the lexicographically largest vertex in $V$.

Next, let $0<s<\dim_H \wus_p$. We will construct a subset $\mathbf{Y}$ of $\ws_q$ such that $\dim_H\mathbf{Y}\geq s$. Since the Hausdorff dimension of a subshift of finite type {is} given by its entropy, we have
\begin{equation}
s<\dim_H \wus_p=\dim_H \wvs_p=\log_2\gamma.
\label{eq:s-small-enough}
\end{equation}

Let $(m_j)_{j\in\N}$ be any strictly increasing sequence of positive integers with $m_1>l$ such that $\alpha_1\dots\alpha_{m_j}^-$ is admissible for each $j$. We claim that for each $j$ there exists a connecting block $b_1\ldots b_{n_j}$ such that 
$\al_1\ldots \al_{m_j}^-b_1\ldots b_{n_j}{\mathbf u}$ is an allowed word in $\wus_q$. This follows essentially from the proof of \cite[Proposition 3.17]{AlcarazBarrera-Baker-Kong-2016}, but for the reader's convenience we sketch the main idea.

Set $i_0:=m_j$. Recursively, for $\nu=0,1,2,\dots$, proceed as follows. If $i_\nu<l_0$, then stop; otherwise, let $i_{\nu+1}$ be the largest integer $i$ such that
\[
\alpha_{i_\nu-i+1}\dots\alpha_{i_\nu}=\overline{\alpha_1\dots \alpha_i}^+.
\]
(If no such $i$ exists, set $i_{\nu+1}=0$.) We now argue that
\begin{equation}
i_{\nu+1}<i_\nu \qquad\mbox{for every $\nu$}.
\label{eq:i-nu-decreasing}
\end{equation}
This will follow once we show that $\alpha_1\dots\alpha_k\succ \overline{\alpha_1\dots\alpha_k}^+$ for every $k\geq l_0$. This inequality is clear for $k\geq 2$, since $q>q_{KL}$ implies $\alpha_1>\overline{\alpha_1}$. On the other hand, if $l_0=1$, then $\alpha_1>\la_1\geq\overline{\la_1}^+> \overline{\alpha_1}^+$, yielding the inequality for $k=1$ as well.

In view of \eqref{eq:i-nu-decreasing}, this process eventually stops, say after $N=N(j)$ steps, with $i_N<l_0$. It is easy to check that $\alpha_1\dots\alpha_{i_\nu}^-$ is admissible for each $\nu<N$. Since $q\in\bb_L$ and $\alpha(q)\succ(\alpha_1\dots\alpha_{i_\nu}^-)^\f$, it follows that
\[
\alpha(q)\succ\alpha_1\dots\alpha_{i_\nu}(\overline{\alpha_1\dots\alpha_{i_\nu}}^+)^\infty, \qquad \nu=1,2,\dots,N-1.
\]
Hence there is a positive integer $k_\nu$ such that
\begin{equation}
\alpha(q)\succ\alpha_1\dots\alpha_{i_\nu}(\overline{\alpha_1\dots\alpha_{i_\nu}}^+)^{k_\nu}, \qquad \nu=1,2,\dots,N-1,
\label{eq:representation-of-alpha-q}
\end{equation}
where by $\alpha(q)\succ\beta_1\dots\beta_i$ we mean that $\alpha_1\dots\alpha_i\succ\beta_1\dots\beta_i$. Put
\[
B_\nu:=(\alpha_1\dots \alpha_{i_\nu}^-)^{k_\nu}, \qquad \nu=1,\dots,N-1,
\]
and $b_1\dots b_{n_{j}}:=B_1 B_2\dots B_{N-1}$, where if $N=1$ we take $B_1 B_2\dots B_{N-1}$ to be the empty word.

Since $|\mathbf{u}|=l-1\geq l_0$, it can be verified using the admissibility of $\alpha_1\dots\alpha_{i_\nu}^-$ for each $\nu$ that $\al_1\ldots \al_{m_j}^-b_1\ldots b_{n_j}{\mathbf u}$ is an allowed word in $\wus_q$.
Here we emphasize that the length $n_j$ of the connecting block depends only on $m_j$, since the word $\mathbf{u}$ is fixed throughout.

We now construct sequences $(r_j)$ and $(R_j)$ as follows: set $R_0=m_1+n_1$, and inductively, for $j=1,2,\dots$, we can choose by {\eqref{eq:s-small-enough} and} Lemma \ref{lem:number-of-paths} an integer $r$ large enough so that
\begin{equation}
{(\log_2 \gamma-s)r\geq (R_{j-1}+m_{j+1}+n_{j+1}+l-1)s+(j+2)\log_2 C}
\label{eq:choice-of-rj}
\end{equation}
and
\begin{equation}
\#\mathbf P_r^{\mathbf{u},\mathbf{u}}\geq C^{-1}\gamma^r.
\label{eq:P_r-lower-bound}
\end{equation}
Put 
\[
r_j:=r, \qquad\mbox{and} \qquad R_j:=R_{j-1}+r_j+m_{j+1}+n_{j+1},
\]
to complete the induction step.
We also set
\[
M_j:=\sum_{i=1}^j(m_i+n_i+r_i), \qquad N_j:=M_j+m_{j+1}, \qquad\mbox{for $j\geq 0$}.
\]

Now let $\mathbf{Y}$ be the set of sequences $(y_i)$ in $\Omega_M$ satisfying the following requirements for all $j\geq 0$:
\begin{enumerate}
\item $y_{M_j+1}\dots y_{M_j+m_{j+1}}=\alpha_1\dots\alpha_{m_{j+1}}^-$;
\item $y_{N_j+1}\dots y_{N_j+n_{j+1}}=b_1\dots b_{n_{j+1}}$;
\item {$y_{R_j+1}\dots y_{R_j+l-1}=\mathbf{u}$;}
\item $y_{R_j+l}\dots y_{M_{j+1}+l-1}=$ the word formed by reading the labels of any path of length $r_{j+1}$ in $G$ that starts and ends at $\mathbf{u}$.
\end{enumerate}
{Note that (4) is consistent with (1) despite the overlapping definitions, since for each $j$, $\mathbf{u}$ is a prefix of $\alpha_1\dots\alpha_{m_j}^-$.} By the construction of the connecting block $b_1\dots b_{n_{j+1}}$, the word $y_{M_j+1}\dots y_{M_{j+1}}$ is allowed in $\wus_q$, for each $j$. It now follows easily that $\mathbf{Y}\subset \ws_q$.  

Next, we construct a mass distribution on $\mathbf{Y}$. Let $t_j$ denote the number of words satisfying the requirement of {(4)} above, and note that by \eqref{eq:P_r-lower-bound},
\begin{equation}
t_j\geq C^{-1}\gamma^{r_{j+1}}, \qquad j\geq 0.
\label{eq:t-lower-bound}
\end{equation}
Define a measure $\mu$ on $\mathbf{Y}$ by
\begin{equation}
\mu([y_1\dots y_k])=\frac{\tilde{t}_j(y_1\dots y_k)}{\prod_{i=0}^j t_i}, \qquad\mbox{for $j\geq 0$ and $R_j+l-1\leq k\leq M_{j+1}$},
%\mu([x_1\dots x_{M_j}])=\mu([x_1\dots x_{R_j}])=\prod_{i=0}^{j-1} \frac{1}{t_i}, \qquad j\geq 1,
\label{eq:definition-of-mu}
\end{equation}
where {$[y_1\dots y_k]:=\{(x_i)\in\mathbf{Y}: x_1\dots x_k=y_1\dots y_k\}$ is the cylinder generated by $y_1\dots y_k$, and} $\tilde{t}_j(y_1\dots y_k)$ is the number of paths in $G$ of length $M_{j+1}+l-1-k$ starting at vertex $y_{k-l+2}\dots y_k$ and ending at $\mathbf{u}$. Observe that
\begin{equation}
\tilde{t}_j(y_1\dots y_k)\leq C\gamma^{M_{j+1}+l-1-k}.
\label{eq:t-tilde-bound}
\end{equation}
We complete the definition of $\mu$ by setting $\mu([y_1\dots y_k])=1$ for $k<R_0+l-1$, and
\begin{equation}
\mu([y_1\dots y_k])=\mu([y_1\dots y_{M_{j}}]), \qquad\mbox{for $j\geq 1$ and $M_{j}<k<R_{j}+l-1$}.
\label{eq:definition-of-mu-complete}
\end{equation}
{It is easy to see that Kolmogorov's consistency conditions are satisfied, so that $\mu$ defines a unique mass distribution on $\mathbf{Y}$. We claim that}
\begin{equation}
\mu([y_1\dots y_k])\leq {\tilde{C}}\big(\diam([y_1\dots y_k])\big)^s
\label{eq:mass-distribution-inequality}
\end{equation}
for any $k\in\N$ and any cylinder $[y_1\dots y_k]$, {where $\tilde{C}:=C^2 2^{(R_0+l-1)s}$.
Observe that $\diam([y_1\dots y_k])=2^{-k}$.
It is clearly sufficient to check \eqref{eq:mass-distribution-inequality} for $R_j+l-1\leq k\leq M_{j+1}$, where $j\geq 0$. Assuming $k$ is in this range, \eqref{eq:definition-of-mu} and the estimates \eqref{eq:t-lower-bound}, \eqref{eq:t-tilde-bound} give
\begin{align*}
\log_2\mu([y_1\dots y_k])+ks &\leq (j+2)\log_2 C+\left(M_{j+1}+l-1-k-\sum_{i=1}^{j+1}r_i\right)\log_2\gamma+ks\\
%&\leq (j+2)\log_2 C+\left(M_{j+1}+l-1-\sum_{i=1}^{j+1}r_i\right)\log_2\gamma-(R_j+l-1)(\log_2\gamma -s)\\
&\leq (R_j+l-1)s+(j+2)\log_2 C-\sum_{i=1}^j r_i\log_2\gamma,
\end{align*}
using that $\log_2\gamma>s$ and $M_{j+1}=R_j+r_{j+1}$.
For $j=0$ this last expression reduces to $(R_0+l-1)s+2\log_2 C=\log_2\tilde{C}$. For $j\geq 1$, it can be written as
\begin{align*}
(R_{j-1}+m_{j+1}+n_{j+1}+l-1)s+(j+2)\log_2 C-\sum_{i=1}^{j-1}r_i\log_2\gamma-r_j(\log_2\gamma-s),
\end{align*}
which is $\leq 0$ by \eqref{eq:choice-of-rj}. Thus, in either case, we obtain \eqref{eq:mass-distribution-inequality}.}

By the mass distribution principle, \eqref{eq:mass-distribution-inequality} implies $\dim_H \ws_q\geq\dim_H \mathbf{Y}\geq s$, as required. Finally, since $s<\dim_H \wus_p$ was arbitrary, we conclude that $\dim_H \ws_q\geq \dim_H \wus_p$.

When $\wvs_p$ is not transitive,  $\wvs_p$ contains by \cite[Lemma 5.9]{AlcarazBarrera-Baker-Kong-2016} a transitive subshift $\mathbf{Z}_p$ of finite type with the same entropy $\log\gamma$, and $\alpha(p)\in\mathbf{Z}_p$. Hence the directed graph associated with $\mathbf{Z}_p$ still contains the vertex $\alpha_1\dots\alpha_{l-1}$, and the above argument goes through with $\mathbf{Z}_p$ replacing $\wvs_p$.
\end{proof}

\begin{proof}[Proof of Theorem \ref{main4}] {Note that $\ws_q=\emptyset$ for any $q\in(1, M+1]\setminus\overline{\ub}$.
In view of Propositions \ref{prop:W-q-C} and \ref{prop:W-q-B} it remains to prove the theorem for $q\in\overline{\ub}\setminus(\cb\cup\bb_L)$. Then $q\in\overline{\ub}\cap(q_L,q_R]$ for some relative plateau $[q_L,q_R]$. Assume $J=J_{\mathbf{i}}=[q_L,q_R]$ is the {\em smallest} such plateau, and let its generating word be $a_1\dots a_m$. Then either $q\in\bb(J)$ or $q$ is the left endpoint of $J_{\mathbf{i}j}$ for some $j\in\N$.} Let $\hat{q}:=\hat{\Phi}_J(q)$. Then $\hat{q}\in\bb_L^*$, so using Lemma \ref{lem:W-bridge}, Proposition \ref{prop:W-q-B}, and Proposition \ref{th:characterization of hat-Phi-J}(iii) we obtain
\begin{equation*}
\dim_H \ws_q=\frac{1}{m}\dim_H \ws_{\hat{q}}^*=\frac{1}{m}\dim_H \wus_{\hat{q}}^*
=\frac{h\big(\wus_{\hat{q}}^*\big)}{m\log 2}=\frac{h(\wus_q(J))}{\log 2}.%=\dim_H \wus_q(J).
\end{equation*}
By Lemma \ref{lem:bi-Lipschitz} and Theorem \ref{main1} this implies 
\[
\dim_H\mathcal W_q=\frac{\log 2}{\log q}\dim_H\ws_q=\frac{h(\wus_q(J))}{\log q}=f_-(q),
\]
completing the proof.
\end{proof}

\section*{Acknowledgments}

Allaart was partially sponsored by NWO visitor's travel grant 040.11.647/4701. Allaart furthermore wishes to thank the mathematics department of Utrecht University, and in particular Karma Dajani, for their warm hospitality during a sabbatical visit in the spring of 2018 when much of this work was undertaken. Kong was supported by NSFC No. 11401516. {Kong would like to thank the Mathematical Institute of Leiden University.}

%\bibliographystyle{abbrv}
%\bibliography{Fractal-Expansions}

\end{document}